\documentclass[a4paper,12pt]{amsart}

\usepackage{amssymb}
\usepackage{latexsym}
\usepackage{amsmath}
\usepackage{amscd}
\usepackage{graphicx}
\usepackage{palatino}
\usepackage{xcolor}
\usepackage{amsrefs}

\title[Real analytic direct summand]
{
Real analytic lift of\\
 foliations of Thurston and Tsuboi
}

\author{Teruaki Kitano}
\address{Department of information Systems Science, 
Faculty of Science and Engineering,
Soka University, 
Tangi-machi 1-236, 
Hachioji, Tokyo, 192-8577, Japan}
\email{kitano@soka.ac.jp}

\author{Yoshihiko Mitsumatsu}
\address{Department of Mathematics, Chuo University, 
1-13-27 Kasuga Bunkyo-ku, 
Tokyo, 112-8551, Japan}
\email{yoshi@math.chuo-u.ac.jp}

\author{Shigeyuki Morita}
\address{Graduate School of Mathematical Sciences, 
The University of Tokyo, 
3-8-1 Komaba, 
Meguro-ku, Tokyo, 153-8914, Japan}
\email{shmori314@gmail.com}


\newtheorem{thm}{Theorem}[section]
\newtheorem{prop}[thm]{Proposition}
\newtheorem{lem}[thm]{Lemma}
\newtheorem{cor}[thm]{Corollary}
\theoremstyle{definition}
\newtheorem{definition}[thm]{Definition}

\newtheorem{remark}[thm]{Remark}
\newtheorem{problem}[thm]{Problem}
\newtheorem{question}[thm]{Question}

\begin{document}

\newcommand{\Mg}{\mathcal{M}_g}
\newcommand{\Mgp}{\mathcal{M}_{g,\ast}}
\newcommand{\Mgb}{\mathcal{M}_{g,1}}

\newcommand{\hg}{\mathfrak{h}_{g,1}}
\newcommand{\ag}{\mathfrak{a}_g}
\newcommand{\Ln}{\mathcal{L}_n}

\newcommand{\Sg}{\Sigma_g}
\newcommand{\Sgb}{\Sigma_{g,1}}
\newcommand{\la}{\lambda}

\newcommand{\Symp}[1]{Sp(2g,\mathbb{#1})}
\newcommand{\symp}[1]{\mathfrak{sp}(2g,\mathbb{#1})}
\newcommand{\gl}[1]{\mathfrak{gl}(n,\mathbb{#1})}

\newcommand{\At}[1]{\mathcal{A}_{#1}^t (H)}
\newcommand{\Hq}{H_{\mathbb{Q}}}

\newcommand{\Ker}{\mathop{\mathrm{Ker}}\nolimits}
\newcommand{\Hom}{\mathop{\mathrm{Hom}}\nolimits}
\renewcommand{\Im}{\mathop{\mathrm{Im}}\nolimits}

\newcommand{\Der}{\mathop{\mathrm{Der}}\nolimits}
\newcommand{\Out}{\mathop{\mathrm{Out}}\nolimits}
\newcommand{\Aut}{\mathop{\mathrm{Aut}}\nolimits}
\newcommand{\Q}{\mathbb{Q}}
\newcommand{\Z}{\mathbb{Z}}
\newcommand{\R}{\mathbb{R}}
\newcommand{\C}{\mathbb{C}}

\begin{abstract}
Thurston constructed codimension one foliations on $S^3$ 
thereby proved that the homomorphism
$gv: \pi_3(B\overline{\Gamma}^\infty_1)\rightarrow \R$ induced by
the Godbillon-Vey invariant is surjective. 
By another real analytic construction,
he proved that the homomorphism
$gv: H_3(B\overline{\Gamma}^\omega_1)\rightarrow \R$ is also surjective
where
$B\overline{\Gamma}^\omega_1$ 
is a $K(\pi,1)$ space
by Haefliger.

Tsuboi
proved that the former surjection splits so that
$\pi_3(B\overline{\Gamma}^\infty_1)= \R\oplus \mathrm{Ker}\,gv$.
He further showed that the subgroup of 
$H_3(B\overline{\Gamma}^\infty_1;\Z)$ generated by
all the Thurston's constructions 
coincides with his direct summand $\R$.
In this paper, we prove that Thurston's second surjection splits and
also that the subgroup of 
$H_3(B\overline{\Gamma}^\omega_1;\Z)$ generated by
all the Thurston's cycles is equal to our direct summand
$\R$ which 
is a lift of Tsuboi's one.
To show this, we modify
the arguments of 
Thurston and Tsuboi 
by replacing Reeb components with a real analytic construction.
We prove certain {\it uniqueness} of them by showing acyclicity
of the affine group in the Haefliger group $\pi_1(B\overline{\Gamma}^\omega_1)$.
We also prove the existence of a new kind of characteristic class of foliations in
$H^4(B\overline{\Gamma}^\omega_1;\Z)$.
\end{abstract}

\renewcommand\baselinestretch{1.1}
\setlength{\baselineskip}{16pt}

\newcounter{fig}
\setcounter{fig}{0}

\maketitle


\section{Introduction}
More than $50$ years ago, in a celebrated paper \cite{thurston} Thurston constructed 
plenty of codimension one foliations on $S^3$
by making use of classical hyperbolic geometry governed by the group $\mathit{\mathit{PSL}}(2,\R)$
as well as the Reeb foliation on the solid torus. 
By this he showed that the Godbillon-Vey class is a continuously varying
characteristic class and proved a ground breaking result that the homomorphism
\begin{equation}
gv: \pi_3(B\overline{\Gamma}^\infty_1)\rightarrow \R
\label{eq:GVs}
\end{equation}
induced by the Godbillon-Vey class is surjective.
Here $B\overline{\Gamma}^\infty_1$ denotes Haefliger's classifying space for his $\overline{\Gamma}_1^\infty$ structures.

In the real analytic category, Thurston also proved that the homomorphism
\begin{equation}
gv: H_3(B\overline{\Gamma}^\omega_1;\Z)\rightarrow \R
\label{eq:GV}
\end{equation}
is surjective. The method for this was to glue his first examples as above with similar ones, associated 
to finite covering groups $\mathit{\mathit{PSL}}(2,\R)^{(n)}\ (n=2,3,\ldots)$ of $\mathit{\mathit{PSL}}(2,\R)$,
along the common boundaries which are various linear foliations on the torus, thus avoiding
the use of the Reeb component which is not real analytic. In his paper \cite{thurston}, Thurston
did not describe details of his proof. However, fortunately Brooks wrote a proof in \cite{bott}.

In \cite{tsuboi84}, Tsuboi generalized Thurston's former construction as follows.
He identified building blocks of Thurston's construction as elements of the second homology
group of $\mathit{\mathit{PSL}}(2,\R)$, considered as a discrete group, {\it relative to} its rotation 
subgroup $\mathit{PSO}(2)$. By a very nice argument, he proved that 
\begin{equation}
H_2(\mathit{\mathit{PSL}}(2,\R),\mathit{PSO}(2);\Z)\cong \R.
\label{eq:rel}
\end{equation}
By utilizing this he was able to prove that the above surjective homomorphism splits
so that
$\pi_3(B\overline{\Gamma}^\infty_1)\cong \R\oplus \mathrm{Ker}\,gv$.
Tsuboi further constructed many more $3$-cycles of $B\overline{\Gamma}^\infty_1$
by considering all the finite covering groups $\mathit{\mathit{PSL}}(2,\R)^{(n)}$ of $\mathit{\mathit{PSL}}(2,\R)$ and proving
similar isomorphisms as in \eqref{eq:rel} for all $n$.
He then proved the remarkable
result that 
the subspace of $H_3(B\overline{\Gamma}^\infty_1;\Z)$ generated by all of his cycles
obtained in this way remains the same as his
direct summand so that the Godbillon-Vey invariant is the complete invariant
for such foliations.

In this paper, we enhance the above constructions of Thurston and Tsuboi to show that
Thurston's surjective homomorphism \eqref{eq:GV} also splits.
In our construction, instead of Reeb component
to bound linear foliations on the torus, we slightly enlarge the group of parallel translations of the real line
$\R$, which is canonically isomorphic to $\R$ itself, by adding just two homothety mappings of it to obtain a group isomorphic to
$\R\rtimes \Z^2$ as a subgroup of $\pi_1(B\overline{\Gamma}^\omega_1)$.
Then we show that any linear foliation on the torus, which can be considered
as an element of $H_2(\R;\Z)$, is the boundary of a $3$-chain of the
above group $\R\rtimes \Z^2$ which is uniquely defined up to homology. 
Our construction of these $3$-chains is an
explicit and systematic one and the Godbillon-Vey class vanishes identically
on them at the chain level. By combining these properties, we 
obtain the desired direct summand $\R\subset H_3(B\overline{\Gamma}^\omega_1;\Z)$.
Furthermore we prove that all of Thurston's analytic cycles of $H_3(B\overline{\Gamma}^\omega_1;\Z)$
mentioned above are contained in this direct summand.



\vspace{3mm}
Convention: In this paper, for any group $G$, including $\mathit{\mathit{PSL}}(2,\R)$ and $\R$, $H_*(G;\Z)$ will denote the integral homology group of $G$
considered as a {\it discrete} group.

\section{Statement of the main results}

\begin{thm}
Thurston's surjective homomorphism
$
gv: H_3(B\overline{\Gamma}_1^\omega;\Z)\rightarrow \R
$
splits. 
The associated direct summand $\R\subset H_3(B\overline{\Gamma}_1^\omega;\Z)$
maps isomorphically onto Tsuboi's direct summand $\R\subset H_3(B\overline{\Gamma}_1^\infty;\Z)$
under the natural projection.
\label{th:main}
\end{thm}

We give two proofs of this theorem. One proof described in $\S 5$ is given by constructing
a homomorphism
$$
H_2(\mathit{\mathit{PSL}}(2,\R),\mathit{PSO}(2);\Z)\overset{\mathit{Tsuboi}}{\cong}\R \rightarrow H_3(B\overline{\Gamma}_1^\omega;\Z)
$$
such that its composition with the homomorphism $gv$ above is an isomorphism.
The other proof described in $\S 6$ is closer to Thurston's original analytic construction.
More precisely, we show that the image of the natural homomorphism
$$
H_2(\mathit{\mathit{PSL}}(2,\R)*_{\mathit{SO}(2)} \mathit{SL}(2,\R);\Z) \rightarrow H_3(B\overline{\Gamma}_1^\omega;\Z)
$$
considered by him
is already isomorphic to $\R$ such that the restriction of $gv$ on it gives an isomorphism.

Now in order to treat Thurston's construction systematically, 
Tsuboi defined in \cite{tsuboi84} a group $G_T$ to be the free product of all the $n$-fold covering
groups $\mathit{\mathit{PSL}}(2,\R)^{(n)}\ (n=1,2,\cdots)$ of $\mathit{\mathit{PSL}}(2,\R)$ amalgamated along their maximal compact subgroups
all of which are canonically identified with the rotation group $\mathit{SO}(2)$. There is a natural homomorphism
$$
\rho_T: G_T\rightarrow \mathrm{Diff}_+^\infty S^1.
$$
By using this, he constructed a homomorphism 
\begin{equation}
H_2(G_T,\mathit{SO}(2);\Z)\rightarrow H_3(B\overline{\Gamma}_1^\infty;\Z)
\label{eq:th}
\end{equation} 
and proved that its image is equal to his direct summand (Theorem 1.2 of the above cited paper).
The following is a real analytic enhancement of this result.

\begin{thm}
Tsuboi's homomorphism can be lifted to a homomorphism 
$$
H_2(G_T,\mathit{SO}(2);\Z)\rightarrow H_3(B\overline{\Gamma}_1^\omega;\Z)
$$
and its image coincides with the direct summand $\R$ given in Theorem \ref{th:main}.
\label{th:mainb}
\end{thm}

\begin{remark}
There is a sequence of natural homomorphisms
$$
H_2(G_T;\Z) \overset{\rho_*}{\rightarrow} H_2(\mathrm{Diff}_+^\omega S^1;\Z)\rightarrow
H_3(\widetilde{\mathrm{Diff}}_+^\omega S^1;\Z)\rightarrow H_3(B\overline{\Gamma}_1^\omega;\Z).
$$
The above theorem includes the statement that the image of the above homomorphism is 
the direct summand $\R$ given in Theorem \ref{th:main}. 
\end{remark}

\begin{cor}
All of Thurston's real analytic cycles in $H_3(B\overline{\Gamma}_1^\omega;\Z)$ are contained in 
the direct summand $\R$ given in Theorem \ref{th:main}. It follows that the Godbillon-Vey
invariant is the complete invariant for such foliations.
\end{cor}

As a corollary to Theorem \ref{th:main}, we show the existence of a huge amount of
non-trivial cohomology classes in $H^4(B\overline{\Gamma}^\omega_1;\Z)$.
These classes serve as a new kind of characteristic class of foliations in the following sense.
They vanish on any space
with $\overline{\Gamma}^\omega_1$
structure if its third homology group is 
finitely generated. On the other hand, they detect uniquely divisible (i.e. $\Q$ vector space)
subgroup in the third integral homology group of the space
on which the Godbillon-Vey class
evaluates non-trivially.
The same is true for the smooth case as well.

\begin{thm}
Let $gv: B\overline{\Gamma}_1^\omega\rightarrow K(\R,3)$ be the classifying map of the
Godbillon-Vey class. Then the induced homomorphism
$$
gv^*: H^4(K(\R,3);\Z)\cong \mathrm{Ext}(\R,\Z)\cong \prod_{\lambda\in\mathcal{H}} \R_\lambda
\rightarrow H^4(B\overline{\Gamma}_1^\omega;\Z)
$$
is injective. The same is true if we replace $B\overline{\Gamma}_1^\omega$ with $B\overline{\Gamma}_1^\infty$.
Here $\mathcal{H}$ denotes a Hamel basis of $\R$ considered as a $\Q$ vector space and $\R_\lambda$ denote
copies of the additive group $\R$ indexed by $\lambda\in\mathcal{H}$. 

Thus the left hand side is a huge group.
However, for any $\overline{\Gamma}_1$ structure on any space $X$,
the characteristic classes
in $H^4(X;\Z)$ corresponding to this group all vanish if $H_3(X;\Z)$ is 
finitely generated.
These classes detect uniquely divisible (i.e. $\Q$ vector space) subgroup of  
$H_3(X;\Z)$ on which the Godbillon-Vey class takes non-trivial values.
\label{th:curious}
\end{thm}

\begin{remark}
It is known that there exist open $3$-manifolds which are $K(\Q,1)$ spaces
(cf. Evans-Moser \cite{em} and an article in mathoverflow \cite{mo}). By suitably adapting the construction of one such manifold
to our situation using the results of this paper,
we can construct open $10$-dimensional manifolds $M$ with the following properties for each $M$.

$(i)$\ $\pi_1(M)$ is a countable subgroup of $\mathrm{Diff}_+^\omega \R\subset \pi_1(B\overline{\Gamma}^\omega_1)$
(or more precisely a subgroup of $\widetilde{\mathit{PSL}}(2,\R)_{(2,3)}$ introduced in $\S 5$ below),

$(ii)$\ $H_3(M;\Z)$ contains a subgroup $V$ which is isomorphic to $\Q$, or more generally finite dimensional
$\Q$ vector space, such that the homomorphism $gv: H_3(M;\Z)\rightarrow \R$
induced by the Godbillon-Vey class is injective on $V$. 

\vspace{1mm}
\noindent
By the above Theorem \ref{th:curious}, we see that the homomorphism
$$
H^4(B\overline{\Gamma}^\omega_1;\Z)\rightarrow H^4(M;\Z)
$$
induced by the $\overline{\Gamma}^\omega_1$ structure on $M$ is non-trivial. Thus our new kind of characteristic
class of foliations can be non-trivial on finite dimensional manifolds with $\overline{\Gamma}^\omega_1$ structure
on them. %
Details of our construction of such manifolds will appear in future.
\end{remark}

\begin{remark}
Hurder \cite{hurder} (Corollary 6.6) proved that $H_{2q+1}(B\Gamma^+_q;\Z)$ contains a subgroup 
isomorphic to $\R$ for any $q\geq 1$ (the case $q=1$ is due to Tsuboi \cite{tsuboi84} and the case
$q=2$ is due to Boullay \cite{boullay}, refer to the above cited paper for details). 
This subgroup is detected by a suitable continuously variable characteristic class in 
$H^{2q+1}(B\Gamma^+_q;\R)$ so that it is a direct summand. Then the same argument as in the
proof for Theorem \ref{th:curious} shows that there is an injection
$$
H^{2q+2}(K(\R,2q+1);\Z)\cong\mathrm{Ext}(\R,\Z)\cong \prod_{\lambda\in\mathcal{H}} \R_\lambda
\subset H^{2q+2}(B\Gamma_q^+;\Z).
$$
\end{remark}

\section{Strategy}
Our strategy is simple. 
As already mentioned in $\S 1$, in the construction of Thurston 
he used the Reeb component to 
bound linear foliations on the torus. 
Tsuboi generalized this construction to yield many more examples.
However these constructions are not real analytic.
In order to obtain real analytic cycles, we add just two
homothety mappings of $\R$ to the group of parallel
translations on it and then prove that any linear foliation on the torus bounds
a $3$-chain in this slightly extended group.

To be more precise, first let us fix our notations. Let 
$\mathrm{Diff}_+^\infty S^1$ denote the group of
orientation preserving $C^\infty$ diffeomorphisms of $S^1$ 
and let $\widetilde{\mathrm{Diff}}_+^\infty S^1$ 
be its universal cover. Let $\mathrm{Diff}_+^\omega S^1, \widetilde{\mathrm{Diff}}_+^\omega S^1$
be the subgroups of $\mathrm{Diff}_+^\infty S^1, \widetilde{\mathrm{Diff}}_+^\infty S^1$
consisting of real analytic elements. Now consider the central extension
$$
0\rightarrow \Z\rightarrow \widetilde{\mathrm{Diff}}_+^\omega S^1
\rightarrow \mathrm{Diff}_+^\omega S^1\rightarrow 1.
$$
It has the following sub central extension
$$
0\rightarrow \Z\rightarrow \R
\rightarrow \mathit{SO}(2)\rightarrow 1
$$
where $\mathit{SO}(2)$ is the subgroup of $\mathrm{Diff}_+^\omega S^1$
consisting of rotations. On the other hand, there is a natural identification
$$
\widetilde{\mathrm{Diff}}_+^\omega S^1= \{f\in \mathrm{Diff}_+^\omega \R; Tf=fT\}\quad
\text{where $T(x)=x+1\ (x\in \R)$}
$$ 
and the subgroup $\R\subset \widetilde{\mathrm{Diff}}_+^\omega S^1$ can be written as
$$
\R=\{T_s; s\in \R\}\quad
\text{where $T_s(x)=x+s\ (x\in \R)$}.
$$
Henceforth we identify $\widetilde{\mathrm{Diff}}_+^\omega S^1$ as a subgroup of $\mathrm{Diff}_+^\omega \R$.

Now consider the subgroup
$$
\mathrm{Diff}_+^\omega \R\supset \R^+=\{S_r;r\in \R ^+\}\quad
\text{where $S_r(x)=rx\ (x\in \R)$}.
$$
Thus we have two subgroups
$$
\R=\{T_s; s\in \R\},\R^+=\{S_r; r\in \R^+\}\subset \mathrm{Diff}_+^\omega \R.
$$
\begin{lem}
Elements of the subgroup $\R^+$ act, by conjugations, on the subgroup $\R\subset \mathrm{Diff}_+^\omega \R$
as
$$
S_r T_s S_r^{-1}=T_{rs}.
$$
\label{lem:ST}
\end{lem}
\begin{proof}
Obviously $S_r^{-1}=S_{r^{-1}}$ so that
$$
S_r T_s S_r^{-1}(x)=r(r^{-1}x +s)=x+rs=T_{rs}(x).
$$
\end{proof}
It follows that the semi-direct product $\R\rtimes \R^+$ can be considered as a subgroup of 
$\mathrm{Diff}_+^\omega \R$ (the orientation preserving affine transformation group of $\R$).
In \cite{morita}, the outer endomorphisms of the group $\widetilde{\mathrm{Diff}}_+^\infty S^1$
induced by the conjugation action of homothety transformations $S_{1/n}$ for $n=2,3,\ldots$,
which we denoted by $\varphi_n$,
were introduced. We studied them further in our paper \cite{kmm}
to investigate the non-triviality question of powers of the Euler class.
The present work began when we noticed that these elements may have a role 
here again. 

For later use, we recall the explicit form of $\varphi_n$ 
(as well as its inverse):
\begin{equation}
\varphi_n(f)=S_{\frac{1}{n}}\circ f\circ S_n,\ \varphi^{-1}_n(f)=S_n\circ f\circ S_{\frac{1}{n}}
\quad (f\in \mathrm{Diff}_+^\omega \R).
\label{eq:varphi}
\end{equation}

Let $Z_{(2,3)}\subset\R^+$ be the subgroup generated by two elements $S_2,S_3$. Clearly $Z_{(2,3)}\cong\Z^2$.
We consider the small subgroup
$$
\R\rtimes Z_{(2,3)}\subset \R\rtimes\R^+.
$$

Now recall that Haefliger \cite{haefliger} proved that his classifying space $B\overline{\Gamma}^\omega_1$
for $\overline{\Gamma}^\omega_1$-structures is a $K(\pi,1)$ space. We would like to
call the group $\Gamma_H=\pi_1(B\overline{\Gamma}^\omega_1)$ the {\it Haefliger group}.
$\mathrm{Diff}_+^\omega \R$ is a subgroup of $\Gamma_H$ so that
we can write
$$
\R\rtimes Z_{(2,3)}\subset \R\rtimes\R^+\subset \mathrm{Diff}_+^\omega \R\subset \Gamma_H.
$$

Now we are ready to state our strategy. Any linear foliation on the torus can be described
as a $2$-cycle 
$$
\sigma=1 \wedge s=[(1,s)-(s,1)]\in H_2(\R;\Z)\cong\wedge^2_\Z \R
$$
of the subgroup $\R\subset \R\rtimes\R^+\subset \Gamma_H$.
The crucial point of our work is to show that any such $2$-cycle bounds a $3$-chain in 
the above small subgroup $\R\rtimes Z_{(2,3)}\subset \Gamma_H$, whose relative homology class
in $H_3(\R\rtimes Z_{(2,3)},\R;\Z)$
is uniquely determined, and
furthermore we obtain explicit formula for such $3$-chains {\it systematically}.
This gives rise to the desired direct summand $\R\subset H_3(B\overline{\Gamma}^\omega_1;\Z)=H_3(\Gamma_H;\Z)$.
One more important fact here is that the Godbillon-Vey cocycle given in \cite{mt} vanishes identically
on $\R\rtimes\R^+$. This is because of the form of the above cocycle combined with the obvious 
fact that the second derivative of any affine transformation $ax+b$
vanishes identically.  This implies that the Godbillon-Vey invariant does not change by our closing of the
linear foliations on the torus in a similar way as in the case of Reeb components
which were used in the construction of Thurston and also by Tsuboi.

\begin{remark}
The group $\R\rtimes\R^+$ considered above is nothing other than the group $\mathrm{Aff}_+ (1)$ of
orientation preserving affine transformations of the real line $\R$.
In general, let $\mathrm{Aff}^+ (n)$ denote the group of
orientation preserving affine transformations of $\R^n$. After finishing the main work of
the present paper,
we found an interesting paper \cite{rozhe} of K. Rozhe (C. Roger) in the late $1970$'s
where he proved a general fact that the natural projection $\mathrm{Aff}^+ (n)\rightarrow \mathrm{GL}^+ (n,\R)$
(considered as {\it discrete} groups)
induces an isomorphism
$$
H_*(\mathrm{Aff}^+ (n);\Z)\cong H_*(\mathrm{GL}^+ (n,\R);\Z).
$$
This result enables us to make our argument much clearer than the original.
\label{re:af}
\end{remark}

\section{Preliminary on inner automorphisms of groups}

Let $G$ be a discrete group and let $BG$ be its classifying space. It is a 
$K(G,1)$ space. For any element 
$\gamma\in G$, let
$$
\varphi: G\rightarrow G
$$
be the inner automorphism of $G$ defined by 
$\varphi(\alpha)=\gamma\alpha\gamma^{-1}\, (\alpha\in G)$.
Let $C_*(G)$ be the chain complex of $G$ with integral coefficients
so that $H_*(C_*(G))=H_*(G;\Z)$. Let
$$
B\varphi:BG\rightarrow BG
$$
be the continuous map induced by $\varphi$. As is well known,
this map is homotopic to the identity map. Therefore, the chain map
$$
\varphi_*:C_*(G)\rightarrow C_*(G)
$$
induced by $\varphi$ is chain homotopic to the identity. Namely there exists
a homomorphism
$$
H_\gamma: C_*(G)\rightarrow C_{*+1}(G)
$$
such that
\begin{equation}
\varphi_*-\mathrm{id}=\partial H_\gamma+H_\gamma\partial
\label{eq:H}
\end{equation}
where $\partial$ denotes the boundary map of $C_*(G)$.
In fact, the following formula which {\it realizes}
the homotopy $B\varphi\sim \mathrm{id.}$ simplicially,
gives an explicit expression for $H_\gamma$.

\begin{prop}
For each $n$-simplex $(\alpha_1,\ldots,\alpha_n)\in C_n(G),\  (\alpha_i\in G)$, we define
\begin{align*}
&H_\gamma(\alpha_1,\ldots,\alpha_n)=-(\gamma,\alpha_1,\ldots,\alpha_n)+(\gamma\alpha_1\gamma^{-1},\gamma,\alpha_2,\ldots,\alpha_n)\\
&-(\gamma\alpha_1\gamma^{-1},\gamma\alpha_2\gamma^{-1},\gamma,\alpha_3,\ldots,\alpha_n)+\ldots
+(-1)^n(\gamma\alpha_1\gamma^{-1},\ldots,\gamma\alpha_{n-1}\gamma^{-1},\gamma,\alpha_n)\\
&+(-1)^{n+1}(\gamma\alpha_1\gamma^{-1},\ldots,\gamma\alpha_{n}\gamma^{-1},\gamma)\in C_{n+1}(G).
\end{align*}
Then $H_\gamma$ satisfies the equality \eqref{eq:H}.
\label{prop:H}
\end{prop}
\begin{proof}
Direct computations imply the desired result.
\end{proof}
\begin{remark}
In the next section, we use the case $n=2$:
$$
H_\gamma(\alpha,\beta)=-(\gamma,\alpha,\beta)+(\gamma\alpha\gamma^{-1},\gamma,\beta)-(\gamma\alpha\gamma^{-1},\gamma\beta\gamma^{-1},\gamma)\
(\alpha,\beta\in G).
$$
\end{remark}

\section{Proof of Theorem \ref{th:main}}

As already mentioned in $\S$ 1, Tsuboi \cite{tsuboi84} proved that $H_2(\mathit{\mathit{PSL}}(2,\R),\mathit{PSO}(2);\Z)\cong\R$.
The parameter $s\in\R$ in this isomorphism has the following geometric meaning. 
Any homology class $\sigma$ in this group can be represented as a 
foliated $S^1$-bundle over a compact oriented surface with a boundary 
whose monodromy is contained in $\mathit{\mathit{PSL}}(2,\R)$ and restricts to a rotation on the boundary
so that the restriction of this foliated $S^1$-bundle to the boundary 
is a linear foliation on the torus.
Since the Euler class of this foliated $S^1$-bundle is trivial
(because the base surface has a boundary), the monodromy group
can be lifted to
$\widetilde{\mathit{\mathit{PSL}}}(2,\R)$ whose restriction to the boundary 
is contained in the subgroup $\R\subset \widetilde{\mathit{\mathit{PSL}}}(2,\R)$
consisting of parallel translations of $\R$.
The parameter $s$ of this cycle is the translation length of the boundary linear foliation
as a foliated $S^1$-product over the circle.
In the case when monodromy of the boundary is trivial so that the foliated $S^1$-bundle
extends to the closed surface, the parameter $s$ is the same as the Euler class
evaluated on the fundamental class which is an integer.
More concretely, for a given homology class $\sigma\in H_2(\mathit{PSL}(2,\R),\mathit{PSO}(2);\Z)$,
its parameter $s(\sigma)\in\R$ can be determined as follows. In the setting of group homology,
$\sigma$ is represented by $2k$ elements $\alpha_i,\beta_i\in  \mathit{PSL}(2,\R)\ (i=1,\ldots,k)$ such that
$$
[\alpha_1,\beta_1]\cdots [\alpha_k,\beta_k]=
\partial(\sigma)\in H_1(\mathit{PSO}(2);\Z)=\mathit{PSO}(2)\ (=\R/\Z).
$$
Now choose
lifts $\tilde\alpha_i,\tilde\beta_i\in \widetilde{\mathit{PSL}}(2,\R)$ of the elements
$\alpha_i,\beta_i
$. Then the value
$$
[\tilde\alpha_1,\tilde\beta_1]\cdots [\tilde\alpha_k,\tilde\beta_k]\in \R\subset \widetilde{\mathit{PSL}}(2,\R)
$$
is well defined independent of the choices made and this is the parameter $s(\sigma)$.

To make a real analytic relative $3$-cycle which bounds linear foliations on the torus, we consider
the homology group of the pair $(\R\rtimes Z_{(2,3)},\R)$. Consider the long exact sequence
\begin{equation}
\cdots \rightarrow H_3(\R;\Z)\rightarrow H_3(\R\rtimes Z_{(2,3)};\Z)\rightarrow
H_3(\R\rtimes Z_{(2,3)},\R;\Z)\overset{\partial}{\rightarrow} H_2(\R;\Z)\rightarrow \cdots
\label{eq:les}
\end{equation}

\begin{prop}
The boundary operator $\partial$ gives rise to an isomorphism
$$
\partial: H_3(\R\rtimes Z_{(2,3)},\R;\Z)\overset{\cong}{\rightarrow} H_2(\R;\Z).
$$
\label{prop:b}
\end{prop}

\begin{proof}
We first show that the homomorphism 
$$
\partial: H_3(\R\rtimes Z_{(2,3)},\R;\Z)\rightarrow H_2(\R;\Z)
$$
is surjective. By the exactness of the sequence \eqref{eq:les},
it is enough to show that the homomorphism $H_2(\R;\Z)\rightarrow H_2(\R\rtimes Z_{(2,3)};\Z)$
is trivial. For any $s_1,s_2\in\R$, we denote by $\langle s_1,s_2\rangle$ the
abelian $2$ cycle $(s_1,s_2)-(s_2,s_1)$ of the group $\R$ for simplicity.
Also we use the same letter for
the restriction of $\varphi^{-1}_2$ to the group $\R\rtimes Z_{(2,3)}$
which is nothing other than the conjugation by the element $S_2$.
Then we have
$$
\varphi^{-1}_2\langle s_1,s_2\rangle=\langle 2s_1,2s_2\rangle.
$$
Since $\varphi^{-1}_2$ acts on $\R\rtimes Z_{(2,3)}$ by the conjugate action 
of the element $S_2$, the left hand side, and hence the right hand side as well, is
homologous to $\langle s_1,s_2\rangle$. On the other hand,
it is easy to see that $\langle 2s_1,2s_2\rangle$ is homologous to $2^2\langle s_1,s_2\rangle$.
Hence $2^2-1=3$ times $\langle s_1,s_2\rangle$ is homologous to $0$.
Similarly as above, $\varphi^{-1}_3$ acts on $\R\rtimes Z_{(2,3)}$ by the conjugate action 
of the element $S_3$. Hence we can conclude that $3^2-1=8$ times 
$\langle s_1,s_2\rangle$ is homologous to $0$.
Since $3$ and $8$ are coprime to each other, we can finally conclude
that $\langle s_1,s_2\rangle$ is homologous to $0$ in $\R\rtimes Z_{(2,3)}$,
finishing the proof that the homomorphism $H_2(\R;\Z)\rightarrow H_2(\R\rtimes Z_{(2,3)};\Z)$
is trivial. 

Next we show that the homomorphism $H_3(\R;\Z)\rightarrow H_3(\R\rtimes Z_{(2,3)};\Z)$
is also trivial. This follows from a similar argument as above because
$2^3-1=7$ and $3^3-1=26$ are coprime to each other.

If we put these facts into the long exact sequence \eqref{eq:les}, then we are left
with the following short exact sequence
\begin{equation}
0 \rightarrow H_3(\R\rtimes Z_{(2,3)};\Z)\rightarrow
H_3(\R\rtimes Z_{(2,3)},\R;\Z)
\overset{\partial}{\rightarrow} H_2(\R;\Z)\rightarrow 0.
\label{eq:ses}
\end{equation}
It remains now to prove
$$
H_3(\R\rtimes Z_{(2,3)};\Z)=0.
$$
We consider the Hochschild-Serre spectral sequence. All the relevant $E^2$-terms 
vanish as follows.
$E^2_{0,3}=H_0(Z_{(2,3)};H_3(\R;\Z))=0$ is proved just above. 
The vanishing of both of $E^2_{1,2}=H_1(Z_{(2,3)};H_2(\R;\Z)), E^2_{2,1}=H_2(Z_{(2,3)};H_1(\R;\Z))$
follows from a general argument of Rozhe in \cite{rozhe}. The last piece
$E^2_{3,0}=H_3(Z_{(2,3)};H_0(\R;\Z))=0$ because $Z_{(2,3)}\cong\Z^2$.

This finishes the proof.
\end{proof}

Now we would like to give an explicit formula for the homomorphism
$$
\Psi:=\partial^{-1}: H_2(\R;\Z)\rightarrow H_3(\R\rtimes Z_{(2,3)},\R;\Z)
$$
at the chain level. This is the heart of our construction of analytic cycles
because it plays the role of the Reeb component in the works of Thurston and Tsuboi.
We would like to call elements of the image of the above homomorphism at the chain level
``{\it real analytic caps}" (see Definition \ref{def:ac} below).

To do this, we prepare a technical lemma. In the proof of Proposition \ref{prop:b}
the following fact was used.
For each $s_1,s_2\in\R$, the cycles $\langle 2s_1,2s_2\rangle, \langle 3s_1,3s_2\rangle$
of the abelian group $\R$ are homologous to
$4\langle s_1,s_2\rangle ,\ 9 \langle s_1,s_2\rangle$,
respectively. In fact we have the following formula which we use in our
explicit description of $\Psi$. 

\begin{lem}
In the chain group $C_*(\R;\Z)$, if we set
\begin{align*}
B_2\langle s_1,s_2\rangle&=(s_1,s_1,2s_2)-(s_1,2s_2,s_1)+(2s_2,s_1,s_1)\\
&-2\{(s_2,s_2,s_1)-(s_2,s_1,s_2)+(s_1,s_2,s_2)\}\\
B_3\langle s_1,s_2\rangle&=(s_1,2s_1,3s_2)-(s_1,3s_2,2s_1)+(3s_2,s_1,2s_1)\\
&+(s_1,s_1,3s_2)-(s_1,3s_2,s_1)+(3s_2,s_1,s_1)\\
&-3\{(s_2,2s_2,s_1)-(s_2,s_1,2s_2)+(s_1,s_2,2s_2)\\
&\hspace{7mm} +(s_2,s_2,s_1)-(s_2,s_1,s_2)+(s_1,s_2,s_2)\}
\end{align*}
then we have
\begin{align*}
\partial B_2\langle s_1,s_2\rangle&=-\langle 2s_1,2s_2\rangle+4\langle s_1,s_2\rangle\\
\partial B_3\langle s_1,s_2\rangle&=-\langle 3s_1,3s_2\rangle+9\langle s_1,s_2\rangle.
\end{align*}
\label{lem:B}
\end{lem}

\begin{proof}
Direct computations yield the required formula.
\end{proof}

\begin{prop}
For any two elements $s_1,s_2\in\R$,
if we set
$$
C\langle s_1,s_2\rangle=\{2(H_{S_3}+B_3)-5(H_{S_2}+B_2)\}\langle s_1,s_2\rangle\in C_3(\R\rtimes\R^+;\Z),
$$
then we have
$$
\partial C\langle s_1,s_2\rangle=\langle s_1,s_2\rangle.
$$
\label{prop:psi}
\end{prop}
\begin{proof}
Using Proposition \ref{prop:H}, \eqref{eq:H} and Lemma \ref{lem:B}, we compute
\begin{align*}
\partial C\langle s_1,s_2\rangle&=2\{\partial H_{S_3}+\partial B_3\}\langle s_1,s_2\rangle
-5\{\partial H_{S_2}+\partial B_2\}\langle s_1,s_2\rangle\\
&=2\{\varphi^{-1}_3\langle s_1,s_2\rangle-\langle s_1,s_2\rangle+9\langle s_1,s_2\rangle-\langle 3s_1,3s_2\rangle\}\\
&\hspace{3mm}-5\{\varphi^{-1}_2\langle s_1,s_2\rangle-\langle s_1,s_2\rangle+4\langle s_1,s_2\rangle-\langle 2s_1,2s_2\rangle\}\\
&=(2\cdot 8 -5\cdot 3)\langle s_1,s_2\rangle\\
&=\langle s_1,s_2\rangle.
\end{align*}
This completes the proof.
\end{proof}
By combining Proposition \ref{prop:b} and Proposition \ref{prop:psi}, we obtain the following.
\begin{prop}
The homomorphism 
$$
\Psi: H_2(\R;\Z)\rightarrow H_3(\R\rtimes Z_{(2,3)},\R;\Z)
$$
is given by the correspondence
$$
H_2(\R;\Z)\ni [\langle s_1,s_2\rangle]\mapsto [C\langle s_1,s_2\rangle]\in H_3(\R\rtimes Z_{(2,3)},\R;\Z)
$$
at the chain level.
\end{prop}

\begin{definition}[Real analytic cap]
For any $2$-cycle $\langle s_1,s_2\rangle\in Z_2(\R;\Z)$, we call $C\langle s_1,s_2\rangle\in Z_3(\R\rtimes Z_{(2,3)},\R;\Z)$
a real analytic cap of $\langle s_1,s_2\rangle$.
\label{def:ac}
\end{definition}

Let $\widetilde{\mathit{PSL}}(2,\R)$ denote the universal covering group of $\mathit{PSL}(2,\R)$.
We consider $\mathit{PSL}(2,\R)$ as a subgroup of $\mathrm{Diff}_+^\omega S^1$ in a way that its subgroup
$\mathit{PSO}(2)$ is identified with $\mathit{SO}(2)\subset \mathrm{Diff}_+^\omega S^1$.
Also we consider $\widetilde{\mathit{PSL}}(2,\R)$ as a subgroup of $\mathrm{Diff}_+^\omega \R$
through the embedding $\widetilde{\mathrm{Diff}}_+^\omega S^1\subset \mathrm{Diff}_+^\omega \R$
mentioned in $\S$ 3. Recall that $\R\rtimes Z_{(2,3)}\subset \R\rtimes\R^+\subset \mathrm{Diff}_+^\omega \R$
is the subgroup generated by the translation group $\R$ and two elements $S_2,S_3\in \R^+$.
We set
$$
\widetilde{\mathit{PSL}}(2,\R)_{(2,3)}=\widetilde{\mathit{PSL}}(2,\R) *_{\R} (\R\rtimes Z_{(2,3)}).
$$
Then there is a canonical (presumably injective) homomorphism
\begin{equation}
\rho: \widetilde{\mathit{PSL}}(2,\R)_{(2,3)}\rightarrow \mathrm{Diff}_+^\omega \R\subset \Gamma_H.
\label{eq:rho}
\end{equation}

We consider the following commutative diagram.
\begin{equation}
\begin{CD}
 H_2(\mathit{PSL}(2,\R),\mathit{PSO}(2);\Z) @>{\mu}>{\text{Gysin map}}> H_3(\widetilde{\mathit{PSL}}(2,\R),\R;\Z) @. \\
 @V{\partial}VV   @V{\partial}VV  @.\\ 
  H_1(\mathit{PSO}(2);\Z) @>{\mu}>{\text{Gysin map}}> H_2(\R;\Z) @>{\Psi=\partial^{-1}}>\cong> H_3(\R\rtimes Z_{(2,3)},\R;\Z).
\end{CD}
\label{eq:cdg}
\end{equation}
Here the upper homomorphism $\mu$ is a part of the {\it{relative}} Gysin exact sequence of the pair of
the central extension $(\widetilde{\mathit{PSL}}(2,\R),\R)\rightarrow (\mathit{PSL}(2,\R),\mathit{PSO}(2))$ and the lower
$\mu$ is its image under the boundary map. In both cases, a homology class $\sigma$ on the left hand side is
represented by a foliated $S^1$-bundle with prescribed monodromy and its image under the homomorphism $\mu$,
which we denote by $\tilde{\sigma}$,
is represented by the associated total space which is a foliated $S^1$-product
with canonical trivialization as a smooth $S^1$-bundle.
Based on this, we make the following definition.

\begin{definition}
We define a homomorphism
$$
\Phi: H_2(\mathit{PSL}(2,\R),\mathit{PSO}(2);\Z)\cong \R \rightarrow H_3(\widetilde{\mathit{PSL}}(2,\R)_{(2,3)};\Z)\rightarrow H_3(\Gamma_H;\Z)
$$
as follows. For any relative homology class $\sigma\in H_2(\mathit{PSL}(2,\R),\mathit{PSO}(2);\Z)$
we set
$$
\Phi(\sigma)=\tilde{\sigma}-\Psi(\partial\tilde{\sigma})\in H_3(\widetilde{\mathit{PSL}}(2,\R)_{(2,3)};\Z)\
\text{and then in}\  H_3(\Gamma_H;\Z).
$$
\label{def:Phi}
\end{definition}

Here the expression $\tilde{\sigma}-\Psi(\partial\tilde{\sigma})$ is a somewhat simplified notation
and its precise meaning will be given in the proof of the following Proposition.
\begin{prop}
The above definition is well defined.
\label{prop:Phi}
\end{prop}

\begin{proof}
Choose a chain $c\in C_3(\widetilde{\mathit{PSL}}(2,\R);\Z)$ which represents the relative cycle
$\tilde{\sigma}\in H_3(\widetilde{\mathit{PSL}}(2,\R),\R;\Z)$. This means that $\partial c\in C_2(\R;\Z)$,
which is obviously a cycle of the chain complex $C_*(\R;\Z)$, represents 
$\partial\tilde{\sigma}=[\partial c]\in H_2(\R;\Z)$.
Now choose a chain $d\in C_3(\R\rtimes Z_{(2,3)};\Z)$ which represents the relative cycle
$\Psi(\partial \tilde{\sigma})\in H_3(\R\rtimes Z_{(2,3)},\R;\Z)$. This means that
$\partial d\in C_2(\R;\Z)$, which is obviously a cycle of the chain complex $C_*(\R;\Z)$, represents 
$\partial\Psi(\partial \tilde{\sigma})\in H_2(\R;\Z)$. By the definition of the homomorphism $\Phi$,
we have $\partial\Psi(\partial \tilde{\sigma})=\partial \tilde{\sigma}$.
Thus both of the two chains $\partial c, \partial d$ represent the same class $\partial\tilde{\sigma}\in H_2(\R;\Z)$.

Here we look into this class more precisely. We have the identity $\partial\tilde{\sigma}=\widetilde{\partial\sigma}$.
Geometrically, $\partial\sigma$ is represented by a foliated $S^1$-bundle over the
circle $S^1$ whose 
monodromy is an element of the rotation group $\mathit{PSO}(2)$. Therefore $\widetilde{\partial\sigma}$ is a
torus equipped with a linear foliation. As an element of $H_2(\R;\Z)\cong\wedge^2_\Z \R$,
it is represented by $1\wedge s\ =[(1,s)-(s,1)]$ where $s\in\R$ denotes the parameter of $\sigma$. 
(If $s$ is an integer, then $1\wedge s=0$.)
By taking appropriate subdivisions of the chains $c,d$,
we may assume that the two chains $\partial c, \partial d$, which are cell decompositions of a torus,
coincide to each other. Now we consider $c, d$ as $3$-chains of the group 
$\widetilde{\mathit{PSL}}(2,\R)_{(2,3)}=\widetilde{\mathit{PSL}}(2,\R)*_\R (\R\rtimes Z_{(2,3)})$ through the inclusions
$\widetilde{\mathit{PSL}}(2,\R), \R\rtimes Z_{(2,3)}\subset \widetilde{\mathit{PSL}}(2,\R)_{(2,3)}$. Then, since $\partial c=\partial d$
(as cell decompositions of a torus) by the assumption, $c-d$ is a $3$-cycle of the group $\widetilde{\mathit{PSL}}(2,\R)_{(2,3)}$.
Our homology class $\tilde{\sigma}-\Psi(\partial\tilde{\sigma})$ is defined to be the class $[c-d]$.

We check that the above class is determined independent of various choices of chains in the above construction.
This is because of the following reason. The long exact sequence of the homology of the pair
$(\widetilde{\mathit{PSL}}(2,\R)_{(2,3)},\R)=(\widetilde{\mathit{PSL}}(2,\R)*_\R \R\rtimes Z_{(2,3)},\R)$ is given by
$$
\cdots\rightarrow H_3(\R;\Z)\overset{\text{$0$-map}}{\rightarrow}
H_3(\widetilde{\mathit{PSL}}(2,\R)_{(2,3)};\Z)
\rightarrow H_3(\widetilde{\mathit{PSL}}(2,\R)_{(2,3)},\R;\Z)\overset{\partial}{\rightarrow} H_2(\R;\Z) \rightarrow \cdots.
$$
Here the fact that the first homomorphism is the $0$-map was proved in the proof of Proposition \ref{prop:b}.
By the excision, we have
$$
H_3(\widetilde{\mathit{PSL}}(2,\R)_{(2,3)},\R;\Z)\cong H_3(\widetilde{\mathit{PSL}}(2,\R),\R;\Z)\oplus
H_3(\R\rtimes Z_{(2,3)},\R;\Z).
$$
Therefore, we can conclude that
\begin{equation}
\begin{split}
H_3(\widetilde{\mathit{PSL}}&(2,\R)_{(2,3)};\Z)\cong \\
\mathrm{Ker}&\left(\partial: H_3(\widetilde{\mathit{PSL}}(2,\R),\R;\Z)\oplus
H_3(\R\rtimes Z_{(2,3)},\R;\Z)
\rightarrow H_2(\R;\Z)\right).
\end{split}
\label{eq:ker}
\end{equation}
The boundary operator $\partial: H_3(\widetilde{\mathit{PSL}}(2,\R),\R;\Z)\oplus
H_3(\R\rtimes Z_{(2,3)},\R;\Z)
\rightarrow H_2(\R;\Z)$
is given by
$$
\partial([c],[d])=[\partial(c)-\partial(d)]\in H_2(\R;\Z).
$$
Now in the Definition \ref{def:Phi}, the pair $(\tilde{\sigma},\Psi(\partial\tilde{\sigma}))$
belongs to the kernel of the above boundary operator $\partial$. Hence the difference
$\tilde{\sigma}-\Psi(\partial\tilde{\sigma})$ defines a unique element
in \\$H_3(\widetilde{\mathit{PSL}}(2,\R)_{(2,3)};\Z)$ by Equality \eqref{eq:ker}.
This completes the proof.
\end{proof}

To proceed to the proof of Theorem \ref{th:main}, we replace $\mathit{PSL}(2,\R)$ with its double
cover $\mathit{SL}(2,\R)$ in Definition \ref{def:Phi}. Here we recall that Tsuboi \cite{tsuboi84} proved
that 
$$
H_2(\mathit{PSL}(2,\R)^{(n)},\mathit{SO}(2);\Z)\cong\R \quad (n=1,2,\cdots)
$$
where the meaning of the parameter $s\in\R$ for $n>1$ is similar to the case
$n=1$. Namely, in the description of the meaning of $s$ for the case $n=1$ 
given in the beginning of this section, we just replace $\mathit{PSL}(2,\R)$ by $\mathit{PSL}(2,\R)^{(n)}$.
Observe that the center of $\widetilde{\mathit{PSL}}(2,\R)^{(n)}$ is the same as 
the center $\Z$ of $\widetilde{\mathit{PSL}}(2,\R)$ for all $n$.
Also we have the following commutative diagram
$$
\begin{CD}
\widetilde{\mathit{PSL}}(2,\R)  @>{\subset}>> \widetilde{\mathrm{Diff}}_+^\omega S^1 @>{\subset}>>\Gamma_H \\
 @V{\varphi_2}V{\cong}V   @V{\varphi_2}V{}V  @V{\varphi_2}V{}V\\ 
\widetilde{\mathit{SL}}(2,\R) @>{\subset}>> \widetilde{\mathrm{Diff}}_+^\omega S^1 @>{\subset}>> \Gamma_H.
\end{CD}
$$
Namely the universal cover $\widetilde{\mathit{SL}}(2,\R)$ of $\mathit{SL}(2,\R)$, considered as a subgroup of
$\widetilde{\mathrm{Diff}}_+^\omega S^1$, is equal to the image under $\varphi_2$ of $\widetilde{\mathit{PSL}}(2,\R)$
(cf. \cite{morita}). Since the action of $\varphi_2$ is equal to the conjugation by the element $S_2^{-1}$,
we can conclude that $\widetilde{\mathit{SL}}(2,\R)$ is contained in the group $\widetilde{\mathit{PSL}}(2,\R)_{(2,3)}$.
Thus we obtain the following definition.

\begin{definition}
We define a homomorphism
$$
\Phi_2: H_2(\mathit{SL}(2,\R),\mathit{SO}(2);\Z)\cong \R \rightarrow H_3(\widetilde{\mathit{PSL}}(2,\R)_{(2,3)};\Z)\rightarrow H_3(\Gamma_H;\Z)
$$
as follows. For any relative homology class $\sigma\in H_2(\mathit{SL}(2,\R),\mathit{SO}(2);\Z)$
we set
$$
\Phi_2(\sigma)=\tilde{\sigma}-\Psi(\partial\tilde{\sigma})\in H_3(\widetilde{\mathit{PSL}}(2,\R)_{(2,3)};\Z)\
\text{and then in}\  H_3(\Gamma_H;\Z).
$$
\label{def:Phi2}
\end{definition}

\begin{prop}
The above definition is well defined.
\label{prop:Phi2}
\end{prop}

\begin{proof}
Observe that
$\widetilde{\mathit{SL}}(2,\R)\subset \widetilde{\mathit{PSL}}(2,\R)_{(2,3)}$ and furthermore that
$$
\widetilde{\mathit{SL}}(2,\R)*_\R (\R\rtimes Z_{(2,3)})=\widetilde{\mathit{PSL}}(2,\R)_{(2,3)}.
$$
Then the proof can be given exactly the same way as that of Proposition \ref{prop:Phi}
by replacing $\widetilde{\mathit{PSL}}(2,\R)$ with $\widetilde{\mathit{SL}}(2,\R)$.
\end{proof}

\begin{prop}
For any parameter $s\in\R$, we have the following identity
$$
\Phi_2(s)=4 \Phi(s)
$$
where $\Phi(s)$ denotes $\Phi(\sigma)$ such that the parameter of $\sigma$ is equal to $s$ and similarly
for $\Phi_2$.
\label{prop:comp}
\end{prop}
\begin{proof}
Since
$$
\varphi_2(1\wedge s)=\frac{1}{2}\wedge \frac{s}{2}=\frac{1}{4} 1\wedge s,
$$
we have
$$
4\varphi_2(\Phi(s))=\Phi_2(s).
$$
On the other hand, $\varphi_2$ induces an inner automorphism of the target group $\widetilde{\mathit{PSL}}(2,\R)_{(2,3)}$.
Therefore $\varphi_2(\Phi(s))=\Phi(s)$. The required result follows from this.
\end{proof}

\vspace{2mm}
\noindent
{\it Proof of Theorem \ref{th:main}}\\
In \cite{tsuboi84}, Tsuboi defined two homomorphisms
\begin{align*}
&t_1: H_2(\mathit{PSL}(2,\R),\mathit{PSO}(2);\Z)\cong \R \rightarrow H_3(B\overline{\Gamma}^\infty_1;\Z)\\
&t_2: H_2(\mathit{SL}(2,\R),\mathit{SO}(2);\Z)\cong \R \rightarrow H_3(B\overline{\Gamma}^\infty_1;\Z)
\end{align*}
and proved the following facts.
\begin{align}
(i)&\, \text{$\mathrm{Im}\, t_1$ is a direct summand of the target and $gv$ is an isomorphism on it}\\
(ii)&\, \text{$t_2(s)=4t_1(s)$ for any $s\in\R$ so that $\mathrm{Im}\, t_2=\mathrm{Im}\, t_1$}.
\label{eq:tsu}
\end{align}
Our task is to compare our homomorphism $\Phi$ (see Definition \ref{def:Phi})
with Tsuboi's homomorphism $t_1$ mentioned above. If we can show directly that the homomorphism $\Phi$
followed by the projection $H_3(\Gamma_H;\Z)\rightarrow H_3(B\overline{\Gamma}^\infty_1;\Z)$ coincides with $t_1$,
then we are done. In other words, we have to show that the $\Gamma_1$-structure of our 
real analytic chain $C(\sigma)$ which bounds linear foliations on the torus is homotopic
to that of Thurston's and Tsuboi's constructions using Reeb foliations.
This seems to be a rather delicate problem. Instead of showing this directly, we adopt an alternate method
where we compare the two constructions {\it{via}} Thurston's real analytic cycles as follows.

For each parameter $s\in\R$, we consider the difference $\Phi_2(s)-\Phi(s)$. At the chain level,
the added $3$-chains $C\langle 1, s\rangle$ of the group $\R\rtimes Z_{(2,3)}$ cancel so that
only chains of the two groups $\widetilde{\mathit{PSL}}(2,\R),\widetilde{\mathit{SL}}(2,\R)$ remain which make a
cycle of the group $\widetilde{\mathit{PSL}}(2,\R)*_{\R} \widetilde{\mathit{SL}}(2,\R)$. This is exactly
the cycle that Thurston adopted to show that the homomorphism \eqref{eq:GV} is surjective.
On the other hand, in the corresponding difference $t_2(s)-t_1(s)$ in Tsuboi's construction,
the added part containing Reeb foliations also cancel and the remaining cycle is Thurston's
cycle projected to $H_3(B\overline{\Gamma}^\infty_1;\Z)$.

Now by Proposition \ref{prop:comp}, $\Phi_2(s)-\Phi(s)=3\Phi(s)$. Similarly by \eqref{eq:tsu},
$t_2(s)-t_1(s)=3t_1(s)$. It follows that $3\Phi(s)\in H_3(B\overline{\Gamma}^\omega_1;\Z)$
projects to $3t_1(s)\in H_3(B\overline{\Gamma}^\infty_1;\Z)$.
Since both of these two classes are members of uniquely divisible groups 
$\mathrm{Im}\, \Phi, \mathrm{Im}\, t_1$, we can now conclude that $\Phi(s)$ projects to
$t_1(s)$ for all $s$. This is because in any uniquely divisible group,
the equality $3x=3y$ implies $x=y$. This completes the proof.
\qed

\begin{remark}
Thurston computed the Godbillon-Vey invariant of his analytic cycle $t_2(s)-t_1(s)$ to be 
equal to $3 gv(t_1(s))$ as described by Brooks in Appendix of \cite{bott}.
The fact that $gv(\Phi(s))=gv(t_1(s))$, which follows from Theorem \ref{th:main}, might be 
shown directly using the formula for the Godbillon-Vey class given in
\cite{mt} because the Godbillon-Vey cocycle given there vanishes constantly on the group $\R\rtimes\R^+$.
However we did not develop technical details for this.
\end{remark}

\section{Alternative proof of Theorem \ref{th:main}}

In this section, we give an alternative proof of our main result Theorem \ref{th:main}. 
In the previous section $\S 5$, we proved 
this theorem by
constructing a homomorphism
$$
H_2(\mathit{PSL}(2,\R),\mathit{PSO}(2);\Z)\overset{\mathit{Tsuboi}}{\cong}\R \rightarrow H_3(B\overline{\Gamma}_1^\omega;\Z)
$$
such that its composition with the homomorphism
$
gv: H_3(B\overline{\Gamma}_1^\omega;\Z)\rightarrow \R
$ 
is an isomorphism.
Our alternative proof of Theorem \ref{th:main} is given along the line of
Thurston's original construction. He considered the group
\begin{equation}
\Gamma_T:=\mathit{PSL}(2,\R)*_{\mathit{SO}(2)} \mathit{SL}(2,\R)
\label{eq:TG}
\end{equation}
which acts naturally on the circle so that there is a homomorphism
$$
\rho_T: \Gamma_T\rightarrow \mathrm{Diff}_+^\omega S^1.
$$
This induces the following associated homomorphism
$$
\tilde{\rho}_T: H_2(\Gamma_T;\Z)\rightarrow H_2(\mathrm{Diff}_+^\omega S^1;\Z)\rightarrow
H_3(\widetilde{\mathrm{Diff}}_+^\omega S^1;\Z)\rightarrow H_3(B\overline{\Gamma}_1^\omega;\Z).
$$
We prove the following stronger result which immediately implies Theorem \ref{th:main}.

\begin{thm}
The image of the homomorphism
$$
\tilde{\rho}_T: H_2(\Gamma_T;\Z)\rightarrow  H_3(B\overline{\Gamma}_1^\omega;\Z)
$$
is isomorphic to $\R$ which maps bijectively onto Tsuboi's direct summand
$\R\subset H_3(B\overline{\Gamma}_1^\infty;\Z)$ under the natural projection.
\label{th:maint}
\end{thm}

This shows that Thurston's surjection $H_3(B\overline{\Gamma}_1^\omega;\Z)\rightarrow\R$
is already achieved by considering only the group $\Gamma_T=\mathit{PSL}(2,\R)*_{\mathit{SO}(2)} \mathit{SL}(2,\R)$
without considering higher covering groups $\mathit{PSL}(2,\R)^{(n)}$ for $n\geq 3$.

As a byproduct of this theorem, we show another result (see Theorem \ref{th:cao} below)
which may have
independent meaning.
Here is the background for this result. 
Recall that we have two characteristic classes
\begin{align*}
& \chi\in H^2(B\mathrm{Diff}_+^{\infty,\delta} S^1;\Z)  \\
& \alpha \in H^2(B\mathrm{Diff}_+^{\infty,\delta} S^1;\R)
\end{align*}
for foliated $S^1$-bundles (see \cite{morita} e.g.). The former is the Euler class
while the latter is the Godbillon-Vey class integrated along the fiber.
These two classes are also defined as cohomology classes of the group $\Gamma_T$
through the homomorphism $\rho_T$.
Although Thurston did not mention in his paper \cite{thurston},
by his construction of real analytic cycles mentioned in $\S 1$, he essentially proved that
the homomorphism
$$
(\chi,\alpha): H_2(B\mathrm{Diff}_+^{\omega,\delta} S^1;\Z)\rightarrow \Z\oplus\R
$$
is surjective. In the proof of the following theorem, we also describe a detailed proof 
of this result of Thurston.

\begin{thm}
Thurston's surjective homomorphism
$$
(\chi,\alpha): H_2(B\mathrm{Diff}_+^{\omega,\delta} S^1;\Z)\rightarrow \Z\oplus\R
$$
splits. Namely, there exists a direct summand $\Z\oplus\R\subset H_2(B\mathrm{Diff}_+^{\omega,\delta} S^1;\Z)$
on which the homomorphism $(\chi,\alpha)$ gives an isomorphism.
\label{th:cao}
\end{thm}

\begin{remark}
For the moment, we have to assume the Axiom of Choice (AC) in the present proof of the above Theorem \ref{th:cao}.
It should be desirable to give a proof without it.
\end{remark}

The proofs of Theorem \ref{th:maint} and Theorem \ref{th:cao} are based on the following Proposition \ref{prop:ca} and the fact that the inclusion 
$\R\subset \widetilde{\mathit{PSL}}(2,\R)*_{\R}(\R\rtimes Z_{(2,3)})$ is homologically trivial 
up to degree $3$ which we showed in $\S 5$.

Recall that Sah and Wagoner \cite{sw} proved that there is an isomorphism
$$
H_2(\mathit{SL}(2,\R);\Z)\cong K^0_2(\R)\oplus \Z
$$
where $K^0_2(\R)$ denotes the second algebraic $K_2$ group of the real field $\R$
omitting the factor $\Z/2$ so that it
is a $\Q$ vector space, and the second factor is detected by the Euler class.
Tsuboi \cite{tsuboi84} extended this theorem to show that there exist similar isomorphisms
$$
H_2(\mathit{PSL}(2,\R)^{(n)};\Z)\cong K^0_2(\R)\oplus \Z
$$
for all the finite covering group $\mathit{PSL}(2,\R)^{(n)}$
of $\mathit{PSL}(2,\R)$. Here the second factors $\Z$ are all detected by the Euler class
and the first factor is the image of the natural homomorphism
$$
H_2(\mathit{SO}(2);\Z) \rightarrow H_2(\mathit{PSL}(2,\R)^{(n)};\Z)
$$
induced by the inclusion $\mathit{SO}(2)\subset \mathit{PSL}(2,\R)^{(n)}$ of the maximal compact subgroup
each of which is canonically identified with the rotation group $\mathit{SO}(2)$.

\begin{prop}
The homomorphism
$$
(\chi,\alpha): H_2(\Gamma_T;\Z) \rightarrow \Z\oplus\R
$$
induced by $(\chi,\alpha)$ 
is surjective and its kernel coincides with the subgroup 
$K^0_2(\R)\subset H_2(\Gamma_T;\Z)$.
Namely we have the following short exact sequence
$$
0\rightarrow K^0_2(\R)\rightarrow H_2(\Gamma_T;\Z)
\overset{(\chi,\alpha)}{\rightarrow}\Z\oplus\R\rightarrow 0.
$$
\label{prop:ca}
\end{prop}

\begin{proof}
 
Let us consider the following long exact sequence for the pair $(\Gamma_T,\mathit{SO}(2))$
\begin{align*}
&\cdots \rightarrow H_2(\mathit{SO}(2);\Z)\overset{i_*}{\rightarrow} H_2(\Gamma_T;\Z)\rightarrow \\
&H_2(\Gamma_T,\mathit{SO}(2);\Z)\cong \R\oplus\R\overset{\partial}{\rightarrow} H_1(\mathit{SO}(2);\Z)\cong\R/\Z\rightarrow0
\end{align*}
where $i: \mathit{SO}(2)\rightarrow \Gamma_T$ denotes the inclusion so that
$$
\mathrm{Im}\, i_*=K^0_2(\R).
$$
By the excision we have an isomorphism
$$
H_2(\Gamma_T,\mathit{SO}(2);\Z)\cong H_2(\mathit{PSL}(2,\R),\mathit{SO}(2);\Z)\oplus H_2(\mathit{SL}(2,\R),\mathit{SO}(2);\Z).
$$
By Tsuboi's result mentioned in $\S 5$, we have an isomorphism
$$
H_2(\mathit{PSL}(2,\R),\mathit{SO}(2);\Z)\oplus H_2(\mathit{SL}(2,\R),\mathit{SO}(2);\Z)\cong \R\oplus\R.
$$
Hence we can conclude
$$
H_2(\Gamma_T;\Z)/K^0_2(\R)\cong \mathrm{Ker}\ \partial
$$
The boundary homomorphism $\partial$ is given by
$$
H_2(\Gamma_T,\mathit{SO}(2);\Z)\cong \R\oplus\R\ni (s_1,s_2)\mapsto s_1-s_2\ \mathrm{mod}\ \Z
$$
It follows that
$$
H_2(\Gamma_T;\Z)/K_2^0(\R)\cong \{(s_1,s_2)\in \R\oplus\R; s_1-s_2\in \Z\}
$$
Now we set
$$
U:=\{(s_1,s_2)\in \R\oplus\R; s_1-s_2\in \Z\}
$$
and study its structure. Let us define
\begin{align*}
\Z_U:&=\left\{\left(\frac{4}{3}n,\frac{1}{3}n\right)\in \R\oplus\R;n\in\Z\right\}\cong\Z\subset U\\
\R_U:&=\{(s,s)\in \R\oplus\R;s\in\R\}\cong\R\subset U
\end{align*}
Then we can see that
$$
U= \Z_U\oplus \R_U.
$$
In fact
$$
U\ni (s_1,s_2)=\left(\frac{4}{3}n,\frac{1}{3}n\right)+\left(s_1-\frac{4}{3}n,s_2-\frac{1}{3}n\right)\quad (s_1-s_2=n).
$$
Hence we can conclude that
$$
H_2(\Gamma_T;\Z)/K^0_2(\R)\cong U=\Z_U\oplus \R_U\cong\Z\oplus\R.
$$
Both of $\chi$ and $\alpha$ vanish on $K^0_2(\R)$ because holonomy of any cycle
belonging to $K^0_2(\R)$ is contained in the rotation group $\mathit{SO}(2)$.
Therefore they define a homomorphism
$$
(\chi,\alpha):H_2(\Gamma_T;\Z)/K^0_2(\R)\cong U \rightarrow \Z\oplus\R
$$
which is given by
\begin{align*}
&\text{$\chi(s_1,s_2)=s_1-s_2$ \ so that\ $\chi\left(\frac{4}{3}n,\frac{1}{3}n\right)=n$\ on $\Z_U$ \ and\ $\chi\equiv 0$\ on \ $\R_U$ }  \\
&\text{$\alpha(s_1,s_2)=-4\pi^2s_1+16\pi^2s_2$\ so that\ $\alpha\equiv 0$ on $\Z_U$\ and\ $\alpha(s,s)=12\pi^2 s$\ on\ $\R_U$}.
\end{align*}
Thus the above direct sum decomposition of the module $U$ as $\Z_U\oplus \R_U$ is the canonically defined unique decomposition
with respect to the two characteristic numbers $\chi$ and $\alpha$.
Here the computation of $\alpha$ on $\R_U$ is essentially due to Thurston as explained by Brooks \cite{bott}
and translated through Tsuboi's decomposition of $H_2(\Gamma_T,\mathit{SO}(2);\Z)$ as above.
In this way, we see that the two numbers give an isomorphism
$$
(\chi,\alpha): H_2(\Gamma_T;\Z)/K^0_2(\R)=\Z_U\oplus\R_U\overset{\cong}{\rightarrow} \Z\oplus\R
$$
This completes the proof.
\end{proof}

\begin{remark}

Instead of relative homology group in the proof of Proposition \ref{prop:ca}, here we study
what will happen if we use the Mayer-Vietoris exact sequence:
 
 \begin{align*}
 \hspace{6mm}H_2(\mathit{SO}(2);\Z)&\rightarrow H_2(\mathit{PSL}(2,\R);\Z)\oplus H_2(\mathit{SL}(2,\R);\Z)\rightarrow\\ 
& H_2(\Gamma_T;\Z)\rightarrow H_1(\mathit{SO}(2);\Z)=\R/\Z\rightarrow 0
 \end{align*}
It is known by Sah-Wagoner and Tsuboi in their papers cited above
$$
H_2(\mathit{PSL}(2,\R);\Z)\cong K^0_2(\R)\oplus\Z_{\mathit{PSL}},\ H_2(\mathit{SL}(2,\R);\Z)\cong K^0_2(\R)\oplus\Z_{SL},
$$
where both $K^0_2(\R)$ come from $H_2(\mathit{SO}(2);\Z)$ and the Euler class $\chi$ sends both direct summand
$\Z_{\mathit{PSL}},\Z_{SL}$ isomorphically onto $\Z$. Also the inclusion $i: \mathit{SO}(2)\subset \Gamma_T$ induces
$
\mathrm{Im}\, i_*=K^0_2(\R)\subset H_2(\Gamma_T;\Z).
$
Therefore we are left with the exact sequence
$$
0\rightarrow \Z_{\mathit{PSL}}\oplus \Z_{SL} \overset{\iota}{\rightarrow}  H_2(\Gamma_T;\Z)/K^0_2(\R)=U 
\overset{\partial}{\rightarrow} \R/\Z\rightarrow 0.
$$
We see from the proof of Proposition \ref{prop:ca} that
$$
\iota(m,n)= (m,0)+(0,-n)=(m,-n)\in U,\quad
(m\in \Z_{\mathit{PSL}}, n\in \Z_{SL}).
$$
\label{re:MV}
\end{remark}

As preparations for the proof of Theorem \ref{th:maint}, we consider the subgroup $K^0_2(\R)\subset H_2(\Gamma_T;\Z)$
and the direct summand
$\Z_U\subset U=H_2(\Gamma_T;\Z)/K^0_2(\R)$. First we consider the former one.
As mentioned above, Sah-Wagoner \cite{sw} proved that
$$
H_2(\mathit{SL}(2,\R);\Z)\cong K^0_2(\R)\oplus\Z.
$$
Here the summand $K^0_2(\R)$ denotes the algebraic $K_2$ group of the real field $\R$
omitting the $\Z/2$-factor and described as follows. Define a subgroup $A\subset \mathit{SL}(2,\R)$ by 
$$
A=\left\{\begin{pmatrix}
r&0\\
0&r^{-1}
\end{pmatrix};
r\in \R^+ \right\}\ \subset \mathit{SL}(2,\R),
$$
which appears in the Iwasawa decomposition $\mathit{SL}(2,\R)=KAN$.
Then 
$$
K^0_2(\R)=\mathrm{Im}\left(H_2(A;\Z)\rightarrow H_2(\mathit{SL}(2,\R);\Z)\right).
$$
Later Tsuboi \cite{tsuboi84} showed that this subgroup coincides with the
image of the second homology group of the maximal compact subgroup $K=\mathit{SO}(2)\subset \mathit{SL}(2,\R)$.
Namely 
$$
K^0_2(\R)=\mathrm{Im}\left(H_2(\mathit{SO}(2);\Z)\rightarrow H_2(\mathit{SL}(2,\R);\Z)\right).
$$
Now, Parry-Sah \cite{ps} (p. 191, (2.40)) asserted the following statement.  In the Gysin exact sequence 
\begin{equation}
\cdots\rightarrow H_4(\mathit{SL}(2,\R);\Z)\overset{\cap \chi}{\rightarrow}H_2(\mathit{SL}(2,\R);\Z)
\overset{\mu}{\rightarrow} H_3(\widetilde{\mathit{SL}}(2,\R);\Z)
\rightarrow H_3(\mathit{SL}(2,\R);\Z)\rightarrow 0
\label{eq:gy}
\end{equation}
corresponding to
the central extension $0\rightarrow \Z\rightarrow \widetilde{\mathit{SL}}(2,\R)\rightarrow \mathit{SL}(2,\R)\rightarrow 1$,
the homomorphism $\cap\chi$ is the $0$-map. It implies that the homomorphism
\begin{equation}
\mu: H_2(\mathit{SL}(2,\R);\Z)\cong K^0_2(\R)\oplus \Z\subset H_3(\widetilde{\mathit{SL}}(2,\R);\Z)
\label{eq:mui}
\end{equation}
is injective.
The authors mentioned that this result {\it ``would follow from the vanishing of the square of $e$ (the Euler class)"}.
This is true for the summand $\Z\subset H_2(\mathit{SL}(2,\R);\Z)$. However 
we think
that the injectivity under $\mu$ in \eqref{eq:mui} of the other
submodule $K^0_2(\R)\subset H_2(\mathit{SL}(2,\R);\Z)$ is {\it not} a consequence of it.
This does not affect our argument below, but see Appendix in this section for this point.

Now we consider the latter direct summand $\Z_U\subset U=H_2(\Gamma_T;\Z)/K^0_2(\R)$.
A homology class $\sigma\in H_2(\Gamma_T;\Z)$ which projects to 
the class $(4,1)\in \Z_U$ with Euler number $3$ was described in our former note \cite{kmm}
as follows.
Let 
$$
\rho:\pi_1(\Sigma_2)\rightarrow \mathit{PSL}(2,\R)\subset \Gamma_T
$$ 
be a Fuchsian representation, corresponding to a hyperbolic structure
on a closed surface $\Sigma_2$ of genus $2$. 
Since $\Sigma_2$ has a spin structure,
this representation lifts to $\mathit{SL}(2,\R)$
so that there is a representation
$$
\rho^s:\pi_1(\Sigma_2)\rightarrow \mathit{SL}(2,\R)\subset\Gamma_T
$$
which projects to $\rho$ under the natural projection $\mathit{SL}(2,\R)\rightarrow \mathit{PSL}(2,\R)$.
Then the Euler numbers are given by $\chi(\rho_*([\Sigma_2]))=-2$ while $\chi(\rho^s_*([\Sigma_2]))=-1$. Now set 
$\sigma=\rho^s_*([\Sigma_2])-2\rho_*([\Sigma_2])\in H_2(\Gamma_T;\Z)$ so that $\chi(\sigma)=3$.
In fact, by Remark \ref{re:MV} it projects to the element $(4,1)\in \Z_U$ as required
and also $\alpha(\sigma)=0$.

We wonder whether an answer to the following question is known or not.
\begin{question}
Let $\mathcal{T}_g$ be the Teichm\"uller space of genus $g$. For each element
$m\in \mathcal{T}_g$, let $\rho_m:\pi_1(\Sigma_g)\rightarrow \mathit{PSL}(2,\R)$
denote the associated faithful discrete representation which is well defined
up to conjugacy. Choose a base point $m\in \mathcal{T}_g$.
Then there is a map
$$
\mathcal{T}_g\rightarrow K^0_2(\R)\subset H_2(\mathit{PSL}(2,\R);\Z)
$$
defined by 
$$
\mathcal{T}_g\ni m\mapsto (\rho_m)_*([\Sigma_g])-(\rho_{m_0})_*([\Sigma_g])
\in K^0_2\subset H_2(\mathit{PSL}(2,\R);\Z).
$$
Is this map the $0$-map or non-trivial? 
\end{question}

It seems natural to expect that the above map is non-trivial and
we tried to prove this. However, it required complicated computations
involving Steinberg symbols so that we could not obtain an answer.

\vspace{5mm}
\noindent
{\it Proof of Theorem \ref{th:maint}}

\vspace{2mm}
Consider the following sequence of homomorphisms
$$
H_2(\Gamma_T;\Z)\overset{\mu}{\rightarrow} H_3(\widetilde{\Gamma}_T;\Z)\rightarrow
H_3(\widetilde{\mathit{PSL}}(2,\R)_{(2,3)};\Z)\rightarrow H_3(\Gamma_H;\Z).
$$
Here recall the very important fact that the central extension 
$\widetilde{\Gamma}_T$
of the Thurston's group $\Gamma_T$ associated with the Euler class is contained in
$\widetilde{\mathit{PSL}}(2,\R)_{(2,3)}$.
Also $\mu$ denotes a part of the Gysin exact sequence. 

\vspace{2mm}
\noindent
{\it First step}

\vspace{2mm}
Here we show $\tilde{\mu}(K^0_2(\R))=0\in H_3(\widetilde{\mathit{PSL}}(2,\R)_{(2,3)};\Z)$
where 
$$
\tilde{\mu}:H_2(\Gamma_T;\Z)\overset{\mu}{\rightarrow} H_3(\widetilde{\Gamma}_T;\Z)\rightarrow
H_3(\widetilde{\mathit{PSL}}(2,\R)_{(2,3)};\Z)
$$ 
denotes the composition of the first two homomorphisms above. 
This is enough to prove Theorem \ref{th:main}.

\vspace{2mm}

The central extension 
$$
0\rightarrow \Z\rightarrow \widetilde{\Gamma}_T\rightarrow \Gamma_T\rightarrow 1
$$
restricts to the central extension
$$
0\rightarrow \Z\rightarrow\R\rightarrow \R/\Z\rightarrow 1
$$
where $\R/\Z$ is the same as the rotation group $\mathit{SO}(2)\subset \Gamma_T$ and
$\R\subset \widetilde{\Gamma}_T$ denotes the set of translations.
By the naturality of the Gysin map, the following diagram is commutative.
$$
\begin{CD}
H_2(\R/\Z;\Z)
 @>{\mu}>> H_3(\R;\Z)
\\
@V{i_*}VV  @VV{j_*}V\\
H_2(\Gamma_T;\Z) @>{\mu}>>H_3(\widetilde{\Gamma}_T;\Z).
\end{CD}
$$
where $j:\R\rightarrow \widetilde{\Gamma}_T$ denotes the inclusion. Now
by Tsuboi
$$
H_2(\Gamma_T;\Z)\supset K^0_2(\R)=\mathrm{Im}\, i_*.
$$
Hence, by the commutativity of the above diagram, we have
$$
\mu(K^0_2(\R))\subset \mathrm{Im}\, j_*.
$$
Now consider the following sequence
$$
H_3(\widetilde{\Gamma}_T;\Z)\rightarrow
H_3(\widetilde{\mathit{PSL}}(2,\R)_{(2,3)};\Z)\rightarrow H_3(\Gamma_H;\Z)
$$
and the commutative diagram
$$
\begin{CD}
H_3(\R)
 @>{=}>> H_3(\R)
\\
@V{j_*}VV  @VV{k_*}V\\
H_3(\widetilde{\Gamma}_T;\Z) @>{\ell_*}>>H_3(\widetilde{\mathit{PSL}}(2,\R)_{(2,3)};\Z).
\end{CD}
$$
By the commutativity, we conclude that
$$
\ell_*(\mu(K^0_2(\R)))\subset k_*(H_3(\R;\Z))
$$
where $\ell: \widetilde{\Gamma}_T\rightarrow \widetilde{\mathit{PSL}}(2,\R)_{(2,3)}$
denotes the inclusion. 
Now the other very important fact here which we proved in $\S 5$ is that
$$
\text{$H_3(\R;\Z)$ vanishes in $H_3(\widetilde{\mathit{PSL}}(2,\R)_{(2,3)};\Z)$}.
$$
The point of the proof is the property that $2^3-1=7$ and $3^3-1=26$ are coprime to each other.
Therefore
the subgroup $K^0_2(\R)\subset H_2(\Gamma_T;\Z)$ vanishes in $H_3(\widetilde{\mathit{PSL}}(2,\R)_{(2,3)};\Z)$
and hence also in $H_3(\Gamma_H;\Z)$.

On the other hand, the quotient $H_2(\Gamma_T;\Z)/K^0_2(\R)$ is isomorphic to
$\Z\oplus \R$ by Proposition \ref{prop:ca}. We can finally conclude that
$$
\mathrm{Im}(H_2(\Gamma_T;\Z)\rightarrow H_3(\Gamma_H;\Z))
$$ 
contains a direct summand $\R$
which maps isomorphically onto $\R$ by the homomorphism gv.
This completes the {\it First step} of the proof of Theorem \ref{th:maint}.

\vspace{3mm}
\noindent
{\it Second step}

\vspace{1mm}
By the first step above, 
the homomorphism $\tilde{\mu}:H_2(\Gamma_T;\Z)\rightarrow H_3(\widetilde{\mathit{PSL}}(2,\R)_{(2,3)};\Z)$
factors through a homomorphism
$$
U:=H_2(\Gamma_T;\Z)/K^0_2(\R)=\Z_U\oplus \R_U\rightarrow H_3(\widetilde{\mathit{PSL}}(2,\R)_{(2,3)};\Z).
$$
Here we prove that the direct summand $\Z_U$ of $U$ vanishes in the above homomorphism.
This will finish the proof of Theorem \ref{th:maint}. We first prove that
the element $\sigma=\rho^s_*([\Sigma_2])-2\rho_*([\Sigma_2])\in H_2(\Gamma_T;\Z)$,
which projects to $3\in \Z_U$, vanishes in $H_3(\widetilde{\mathit{PSL}}(2,\R)_{(2,3)};\Z)$
under the above homomorphism. This proof can be given using only the property of
the endomorphism $\varphi_2:\widetilde{\mathrm{Diff}}_+^\omega S^1\rightarrow\widetilde{\mathrm{Diff}}_+^\omega S^1 $ 
without making use of our technology developed in 
the preceding sections. Recall that $\varphi_2$ sends the subgroup $\widetilde{\mathit{PSL}}(2,\R)\subset\widetilde{\mathrm{Diff}}_+^\omega S^1$
isomorphically onto the subgroup $\widetilde{\mathit{SL}}(2,\R)\subset\widetilde{\mathrm{Diff}}_+^\omega S^1$
(see $\S 5$). Also recall that 
$
\Gamma_T=\mathit{PSL}(2,\R)*_{\mathit{SO}(2)} \mathit{SL}(2,\R)
$
so that
$$
\widetilde{\Gamma}_T=\widetilde{\mathit{PSL}}(2,\R)*_\R \widetilde{\mathit{SL}}(2,\R).
$$
Now, if we denote by $S^1(\rho)$ (resp. $S^1(\rho^s)$) the total space of the flat
$S^1$-bundle induced by $\rho$ (resp. $\rho^s$), 
then there exists a
fiberwise $2$-fold covering map $S^1(\rho^s)\rightarrow S^1(\rho)$. 
Therefore the endomorphism $\varphi_2$ appears in the following commutative diagram.
\begin{equation}
\begin{CD}
\pi_1(S^1(\rho^s))@>{\tilde{\rho^s}}>> \widetilde{\mathit{SL}}(2,\R)
 @>{\subset}>> \widetilde{\mathrm{Diff}}_+^\omega S^1
\\
@V{\cap}V{\text{index $2$}}V @A{\varphi_2}A{\cong}A @AA{\varphi_2}A\\
\pi_1(S^1(\rho)) @>{\tilde{\rho}}>>\widetilde{\mathit{PSL}}(2,\R)@>{\subset}>>\widetilde{\mathrm{Diff}}_+^\omega S^1
\end{CD}
\label{eq:cd}
\end{equation}
where $\tilde{\rho}:\pi_1(S^1(\rho))\rightarrow \widetilde{\mathit{PSL}}(2,\R)$
(resp. $\tilde{\rho^s}:\pi_1(S^1(\rho^s))\rightarrow \widetilde{\mathit{SL}}(2,\R)$)
denotes the representation induced by $\rho$ (resp. $\rho^s$).
Since $S^1(\rho)$ is a closed oriented $3$-manifold, its fundamental class represents a homology
class $[S^1(\rho)]\in H_3(\widetilde{\mathit{PSL}}(2,\R);\Z)$.  In fact, it is the Gysin image of
the class $[\Sigma_2]\in H_2(\mathit{PSL}(2,\R);\Z)$ so that $[S^1(\rho)]=\mu_{\mathit{PSL}}(\rho_*([\Sigma_2]))$. Similarly 
$[S^1(\rho^s)]\in H_3(\widetilde{\mathit{SL}}(2,\R);\Z)$ and $[S^1(\rho^s)]=\mu_{\mathit{SL}}(\rho^s_*([\Sigma_2]))$.
Here $\mu_{\mathit{PSL}}$ denotes the Gysin homomorphism
$$
\mu_{\mathit{PSL}}: H_2(\mathit{PSL}(2,\R);\Z)\rightarrow H_3(\widetilde{\mathit{PSL}}(2,\R);\Z)
$$ 
and $\mu_{\mathit{SL}}$ 
denotes similar one for the group $\mathit{SL}(2,\R)$. Also let $\mu_\Gamma$ denotes the Gysin homomorphism
for the group $\Gamma_T$. 
Now by the naturality of the Gysin homomorphisms for the group $\Gamma_T$ and its subgroups $\mathit{PSL}(2,\R)$ and 
$\mathit{SL}(2,\R)$, we have
\begin{align*}
\mu_\Gamma(\sigma)&=\mu_\Gamma(\rho^s_*([\Sigma_2])-2\rho_*([\Sigma_2]))\\
&=\tilde{\rho^s}_*(\mu_{\mathit{SL}}([\Sigma_2]))-2\tilde{\rho}_*(\mu_{\mathit{PSL}}([\Sigma_2]))\\
&=(i_{\mathit{SL}})_*[S^1(\rho^s)]-2(i_{\mathit{PSL}})_*[S^1(\rho)]\in H_3(\widetilde{\Gamma}_T;\Z)\
\end{align*}
where $i_{\mathit{PSL}}: \widetilde{\mathit{PSL}}(2,\R)\rightarrow \widetilde{\Gamma}_T$ 
(resp. $i_{\mathit{SL}}: \widetilde{\mathit{SL}}(2,\R)\rightarrow \widetilde{\Gamma}_T$)
denotes the inclusion.

Now by the above commutative diagram \eqref{eq:cd}, we have
\begin{equation}
(\varphi_2)_*(2 [S^1(\rho)])=[S^1(\rho^s)].
\label{eq:ss}
\end{equation}
If we consider everything in the larger group $\widetilde{\mathit{PSL}}(2,\R)_{(2,3)}$
which contains $\widetilde{\Gamma}_T$, then we obtain the following commutative diagram
\begin{equation}
\begin{CD}
\widetilde{\mathit{SL}}(2,\R)@>{i_{\mathit{SL}}}>> \widetilde{\Gamma}_T
 @>{j^s}>{\subset}> \widetilde{\mathit{PSL}}(2,\R)_{(2,3)}
\\
@A{\varphi_2}A{\cong}A  @. @A{\cong}A{\varphi_2}A\\
\widetilde{\mathit{PSL}}(2,\R) @>{i_{\mathit{PSL}}}>>\widetilde{\Gamma}_T@>j>{\subset}>\widetilde{\mathit{PSL}}(2,\R)_{(2,3)},
\end{CD}
\label{eq:cd2}
\end{equation}
where $j, j^s$ denote inclusions.
We can now conclude that
\begin{align*}
\tilde{\mu}(\sigma)&=(j^s\circ i_{\mathit{SL}})_*[S^1(\rho^s)]-2(j\circ i_{\mathit{PSL}})_*[S^1(\rho)]\\
&=(j^s\circ i_{\mathit{SL}})_*(\varphi_2)_*(2 [S^1(\rho)])-2(j\circ i_{\mathit{PSL}})_*[S^1(\rho)]\quad \text{(by \eqref{eq:ss})}\\
&=2(\varphi_2)_*(j\circ i_{\mathit{PSL}})_*[S^1(\rho)]-2(j\circ i_{\mathit{PSL}})_*[S^1(\rho)]\quad (\text{commutativity of \eqref{eq:cd2}})\\
&=0.
\end{align*}
The last equality follows from the fact that $\varphi_2$ acts on the group $\widetilde{\mathit{PSL}}(2,\R)_{(2,3)}$
by conjugation.
This finishes the proof that $3\Z_U\subset U$ vanishes under the homomorphism $\tilde{\mu}$.

To proceed to prove the vanishing of the whole $\Z_U\subset U$, we have to enhance the above argument by
employing our technique developed 
in the preceding sections. It is a refinement of the previous one comparing certain $2$-cycles of 
$\mathit{PSL}(2,\R)$ and $\mathit{SL}(2,\R)$ adapted to the {\it relative} $2$-cycles of 
$\mathit{PSL}(2,\R)$ and $\mathit{SL}(2,\R)$ with respect to their maximal compact subgroups.
To do this,
we consider the homomorphisms $\Phi$ and $\Phi_2$
\begin{align*}
\Phi&: H_2(\mathit{PSL}(2,\R),\mathit{PSO}(2);\Z)\cong \R \rightarrow H_3(\widetilde{\mathit{PSL}}(2,\R)_{(2,3)};\Z)\rightarrow H_3(\Gamma_H;\Z)\\
\Phi&_2: H_2(\mathit{SL}(2,\R),\mathit{SO}(2);\Z)\cong \R \rightarrow H_3(\widetilde{\mathit{PSL}}(2,\R)_{(2,3)};\Z)\rightarrow H_3(\Gamma_H;\Z)
\end{align*}
introduced in $\S 5$. We set
$$
\tau=\left(\frac{4}{3},\frac{1}{3}\right)\in U
$$
so that $\chi(\tau)=1, \alpha(\tau)=0$ and $\tau=1\in \Z_U\subset U$. We can write
$\tau=(\tau_1,\tau_2)$ where
\begin{align*}
\tau_1&= \frac{4}{3}\in H_2(\mathit{PSL}(2,\R),\mathit{PSO}(2);\Z)\cong \R\\
\tau_2&= \frac{1}{3}\in H_2(\mathit{SL}(2,\R),\mathit{SO}(2);\Z)\cong \R.
\end{align*}
Then there are compact oriented surfaces with one boundary component of genus $g_i\ (i=1,2)$, denoted by
$\Sigma_{g_i}^0=\Sigma_{g_i}\setminus \mathrm{Int}\, D^2\ (i=1,2)$,
and homomorphisms
\begin{align*}
\rho_1&: \pi_1(\Sigma_{g_1}^0,p_1)\rightarrow \mathit{PSL}(2,\R)\\
\rho_2&: \pi_1(\Sigma_{g_2}^0,p_2)\rightarrow \mathit{SL}(2,\R)
\end{align*}
such that they represent the classes $\tau_1, \tau_2$ respectively,
where $p_i\in \partial \Sigma_{g_i}^0\ (i=1,2)$ denote base points.
Therefore the restrictions of the homomorphisms $\rho_i$ to the boundaries
\begin{align*}
\rho_1|_{\partial \Sigma_{g_1}}&: \pi_1(\partial \Sigma_{g_1}^0,p_1)\cong\Z \rightarrow \mathit{PSO}(2)\\
\rho_2|_{\partial \Sigma_{g_2}}&: \pi_1(\partial \Sigma_{g_2}^0,p_2)\cong \Z\rightarrow \mathit{SO}(2)
\end{align*}
are given by
\begin{align*}
\Z\ni 1\mapsto \frac{1}{3}\ (=\text{$120^\circ$\ rotation})\ \in \mathit{PSO}(2)\cong \R/\Z\\
\Z\ni 1\mapsto \frac{1}{3}\ (=\text{$120^\circ$\ rotation})\ \in \mathit{SO}(2)\cong \R/\Z
\end{align*}
where both $\mathit{PSO}(2)$ and $\mathit{SO}(2)$ are canonically identified with $\R/\Z$.
Furthermore the homomorphisms $\rho_i$ can be lifted to
\begin{align*}
\tilde{\rho}_1&: \pi_1(\Sigma_{g_1}^0,p_1)\rightarrow \widetilde{\mathit{PSL}}(2,\R)\\
\tilde{\rho}_2&: \pi_1(\Sigma_{g_2}^0,p_2)\rightarrow \widetilde{\mathit{SL}}(2,\R)
\end{align*}
whose restrictions to the boundaries 
\begin{align*}
\tilde{\rho}_1|_{\partial \Sigma_{g_1}}&: \pi_1(\partial \Sigma_{g_1}^0,p_1)\cong\Z \rightarrow \Z\subset \widetilde{\mathit{PSO}}(2)\\
\tilde{\rho}_2|_{\partial \Sigma_{g_2}}&: \pi_1(\partial \Sigma_{g_2}^0,p_2)\cong \Z\rightarrow \Z\subset \widetilde{SO}(2)
\end{align*}
satisfy the conditions
\begin{align*}
\Z\ni 1\mapsto \frac{4}{3}\in \Z\subset \widetilde{\mathit{PSO}}(2)\\
\Z\ni 1\mapsto \frac{1}{3}\in \Z\subset \widetilde{SO}(2)
\end{align*}
(the translation lengths).
Now we glue the two compact surfaces $\Sigma_{g_i}^0\ (i=1,2)$ along their boundaries,
matching the base points and orientations, to make a closed oriented surface 
$\Sigma_g\ (g=g_1+g_2)$. Here the orientation on it is given as $\Sigma_{g_1}^0-\Sigma_{g_2}^0$.
Then we can construct a homomorphism
$$
\rho: \pi_1(\Sigma_g,p)\rightarrow \Gamma_T=\mathit{PSL}(2,\R)*_{\mathit{PSO}(2)=\mathit{SO}(2)} \mathit{SL}(2,\R)
$$
where $p\in \Sigma_g$ denotes the base point obtained by glueing $p_1$ and $p_2$.
By the theorem of van Kampen, we have
$$
\pi_1(\Sigma_g,p)=\pi_1(\Sigma_{g_1}^0,p_1)*_{\pi_1(\partial \Sigma_{g_1}^0,p_1)=\pi_1(\partial \Sigma_{g_2}^0,p_2)}\pi_1(\Sigma_{g_2}^0,p_2).
$$
Then we can define $\rho$ by setting
$$
\rho|_{\pi_1(\Sigma_{g_i}^0,p_i)}=\rho_i\ (i=1,2)
$$
because 
$$
\rho_1|_{\partial \Sigma_{g_1}^0}=\rho_2|_{\partial \Sigma_{g_2}^0}
$$
once we identify $\mathit{PSO}(2)=\mathit{SO}(2)$.
Now we set 
$$
\tilde{\tau}=\rho_*([\Sigma_g])\in H_2(\Gamma_T;\Z).
$$
By the above construction, we see that
$$
H_2(\Gamma_T;\Z)\ni \tilde{\tau}\mapsto \tau\in H_2(\Gamma_T;\Z)/K^0_2(\R)
$$
under the natural projection. Now, by the construction of the homomorphism $\rho$,
we see that
$$
\tilde{\mu}(\tilde{\tau})=\tilde{\tau}_1-\tilde{\tau}_2.
$$
Namely, the homology class $\tilde{\mu}(\tilde{\tau})$ is represented by the 
$3$-cycle $\tilde{\tau}_1-\tilde{\tau}_2$.
On the other hand, we can also see that the following equality holds
\begin{equation}
\tilde{\mu}(\tilde{\tau})=\Phi\left(\frac{4}{3}\right)-\Phi_2\left(\frac{1}{3}\right).
\label{eq:mm}
\end{equation}
This is the key fact in our argument. The point here is that the ``real analytic cap"
appearing in the first term is exactly equal to that of the second term so that they cancel
each other in the above expression.
(cf. the commutative diagram \eqref{eq:cdg}). 
More explicitly, this can be seen as follows.
By the definitions of $\Phi$ and $\Phi_2$
(see Definition \ref{def:Phi} and Definition \ref{def:Phi2}), we have
\begin{align*}
\Phi\left(\frac{4}{3}\right)&=\tilde{\tau}_1-\Psi(\partial \tilde{\tau}_1)\\
\Phi_2\left(\frac{1}{3}\right)&=\tilde{\tau}_2-\Psi(\partial \tilde{\tau}_2).
\end{align*}
Since 
$$
\frac{4}{3}\equiv \frac{1}{3}\ \mathrm{mod}\ 1,
$$
we have
$$
1\wedge \frac{4}{3} =1\wedge \frac{1}{3}\in H_2(\R;\Z)\ \Rightarrow \ 
\partial \tilde{\tau}_1=\partial \tilde{\tau}_2.
$$
Hence
$$
\Psi(\partial \tilde{\tau}_1)=\Psi(\partial \tilde{\tau}_2).
$$
This verifies the above cancellation and thereby
confirms the above equality \eqref{eq:mm}.

Now by Proposition \ref{prop:comp}, we know that $\Phi_2(s)=4\Phi(s)$ for any $s$.
We can now conclude that
$$
\tilde{\mu}(\tilde{\tau})=\Phi\left(\frac{4}{3}\right)-\Phi_2\left(\frac{1}{3}\right)
=4\Phi\left(\frac{1}{3}\right)-\Phi_2\left(\frac{1}{3}\right)=0.
$$
Since $\tau$ represents the generator of the summand $\Z_U\subset U$, this completes the proof of {\it Second step}
and hence that of Theorem \ref{th:maint}.
\qed

\vspace{5mm}
\noindent
{\it Proof of Theorem \ref{th:cao} }\\

By Proposition \ref{prop:ca}, there exists a short exact sequence
\begin{equation}
0\rightarrow K^0_2(\R)\rightarrow H_2(\Gamma_T;\Z)
\overset{(\chi,\alpha)}{\rightarrow}\Z\oplus\R\rightarrow 0.
\label{eq:kh}
\end{equation}
It follows that there is a short exact sequence
$$
0\rightarrow K^0_2(\R)\rightarrow \mathrm{Ker}\ \chi\ (\subset H_2(\Gamma_T;\Z))
\overset{\alpha}{\rightarrow}\R\rightarrow 0.
$$
Recall that $K^0_2$ is a $\Q$ vector space. It is easy to see that group $\mathrm{Ker}\ \chi$
is a torsion free divisible group and so that it is a $\Q$ vector space.
Now it is easy to see that the extension \eqref{eq:kh} is a split extension if we assume the Axiom of Choice.
The required result follows from this.
\qed

\begin{problem}
Let 
$
\rho_T: H_2(\Gamma_T;\Z)\rightarrow H_2(\mathrm{Diff}_+^{\omega} S^1;\Z)
$
be the natural homomorphism. Prove (or disprove) that it is
\begin{align*}
&(i)\quad \text{trivial on $K^0_2(\R)\subset H_2(\Gamma_T;\Z)$}\\
&(ii)\quad \text{surjective}
\end{align*}
so that
$$
H_2(\mathrm{Diff}_+^{\omega} S^1;\Z)\cong \Z\oplus \R.
$$
\end{problem}

As for the first question $(i)$ above,
since $K^0_2(\R)$ comes from the inclusion $i: \mathit{SO}(2)\rightarrow \mathrm{Diff}_+^\omega S^1$
by Tsuboi's result,
we may more generally ask the following.

\begin{question} 
 What can be said about the following homomorphism?
 $$
 i_*:H_*(\mathit{SO}(2);\Z)
 \rightarrow H_*(\mathrm{Diff}_+^\omega S^1;\Z).
 $$
 The case of degree $3$ is particularly important because the triviality
 (or non-triviality) of the homomorphism
 $$
 H_3(\mathit{SO}(2);\Z)
 \supset \wedge^3_\Z(\Q/\Z)\cong \Q/\Z\rightarrow H_3(\mathrm{Diff}_+^\omega S^1;\Z)
 $$
 is closely related to the question of non-triviality (or triviality) of $\chi^2\in H^4(\mathrm{Diff}_+^\omega S^1;\Q)$
 (see \cite{kmm}).
\end{question}

\begin{problem}
Determine whether the natural homomorphism
\begin{equation}
H^*_{GF}(S^1,\mathit{SO}(2))\cong\R[\chi,\alpha]/(\chi\alpha)\rightarrow H^*(\mathrm{Diff}_+^\omega S^1;\R)
\label{eq:bh}
\end{equation}
introduced by Bott and Haefliger
is injective or not.
Here the left hand side denotes the Gel'fand-Fuchs cohomology of the circle relative to the
rotation group $\mathit{SO}(2)$ and the isomorphism is due to Haefliger based on the fundamental work of 
Gel'fand and Fuchs.
\label{prob:gf}
\end{problem}

Thurston proved the injectivity of the homomorphism
\eqref{eq:bh} in degree $2$ by considering his group $\Gamma_T$ described above (see \eqref{eq:TG}), and 
in the smooth case,  he also proved
the non-triviality of 
all the powers of $\alpha$ based on results of Mather.
Then it was proved in \cite{morita} that the above homomorphism with the target
replaced by $H^*(\mathrm{Diff}_+^\infty S^1;\R)$ is injective.

However, in the real analytic case, almost nothing is known beyond Thurston's results.
Even the problem of determining non-triviality of the square
of the rational Euler class $\chi$ is already very difficult and that of higher powers of $\chi$ gets more and more
difficult (as for the {\it integral}
Euler class, Nariman \cite{nariman} proved non-triviality of all the powers of it).
Moreover non-triviality of the powers $\alpha^k\ (k=2,3,\ldots)$ of $\alpha$ should be considered as a big mystery.
Certainly completely new ideas must be necessary to challenge this problem.

\vspace{7mm}
\begin{center}
Appendix
\end{center}

The content of this appendix is not used in other part of the present paper.

We 
consider the assertion of Parry-Sah \cite{ps} (p. 191, (2.40)) mentioned above which implies that
the homomorphism
\begin{equation}
\mu: H_2(\mathit{SL}(2,\R);\Z)\cong K^0_2(\R)\oplus \Z\subset H_3(\widetilde{\mathit{SL}}(2,\R);\Z)
\label{eq:mui2}
\end{equation}
is injective.
The authors mentioned that this result {\it ``would follow from the vanishing of the square of $e$ (the Euler class)"}.

This is true for the summand $\Z\subset H_2(\mathit{SL}(2,\R);\Z)$ because by the well known property of the 
Gysin exact sequence, the injectivity under $\mu$ of this summand $\Z$, which is 
detected by the Euler class,
is equivalent to the vanishing of the square of the rational Euler class $\chi_\Q^2$
in $H^4(\mathit{SL}(2,\R);\Q)$.
But this vanishing follows from the fact that the Lie algebra cohomology of 
$\mathit{SL}(2,\R)$ is given as $H^*(\mathfrak{sl}(2,\R),\mathit{SO}(2))\cong \R[\chi]/(\chi^2)$
(cf. Milnor \cite{milnor}).
We can extend this argument to the cases of larger groups such as $\Gamma_T$ and $G_T$
so that the vanishing $\chi^2_\Q=0$ holds in $H^4(\Gamma_T;\Q)$ and $H^4(G_T;\Q)$ as well. 
It is an important problem to
determine whether this continues to hold
in $H^4(\mathrm{Diff}_+^\omega S^1;\Q)$ or not (see Problem \ref{prob:gf} above).

On the other hand however, we think
that the injectivity under $\mu$ in \eqref{eq:mui2} of the other
submodule $K^0_2(\R)\subset H_2(\mathit{SL}(2,\R);\Z)$ is {\it not} a consequence of
the vanishing of the square of the Euler class. A simple counter-example: 
in the $S^1$-bundle $\text{(Hopf map,id)}: S^3\times S^2\rightarrow S^2\times S^2$,
$\chi^2=0$ but
$\mu: H_2(S^2\times S^2;\Z)\rightarrow H_3(S^3\times S^2;\Z)$
is not an injection.
It needs another argument.
The authors gave, what they called, a direct argument
to prove this result. However, they only gave the homomorphism $\mu$ at the cycle level
by constructing $3$-dimensional cycles in the target $H_3(\widetilde{\mathit{SL}}(2,\R);\Z)$ associated to any $2$-dimensional
cycle in $H_2(A;\Z)$ which represents elements of $K^0_2(\R)$ as well as a $2$-dimensional cycle representing
the generator of the summand $\Z\subset H_2(\mathit{SL}(2,\R);\Z)$. The authors called the resulting $3$-cycles 
{\it ``lifts"} of the original $2$-cycles. 
As far as we understand, this does not prove the required results. It is necessary to
show that these cycles are {\it homologically non-trivial}.
We could not find a proof for this injectivity. However, this does not affect our argument for
the proof of Theorem \ref{th:maint}.

\section{Proof of Theorem \ref{th:mainb}}

Tsuboi's homomorphism \eqref{eq:th} mentioned in $\S$ 2, was defined as follows.
First he proved that
\begin{align*}
H_2(G_T,\mathit{SO}(2);\Z)&\cong \bigoplus_{n=1}^\infty H_2(\mathit{PSL}(2,\R)^{(n)},\mathit{SO}(2);\Z)\\
H_2(\mathit{PSL}(2,\R)^{(n)}&,\mathit{SO}(2);\Z)\cong \R\quad (n=1,2,\cdots)
\end{align*}
and then constructed homomorphisms
$$
t_n: H_2(\mathit{PSL}(2,\R)^{(n)},\mathit{SO}(2);\Z)\rightarrow H_3(B\overline{\Gamma}^\infty_1;\Z)
$$
for all $n$. Finally he proved that the equality $t_n(s)=n^2 t_1(s)\, (s\in\R)$ holds for any $n$ which
implies his main theorem of \cite{tsuboi84}.

To make a real analytic lift of the homomorphism $t_n$,
we observe that the following commutative diagram holds:
$$
\begin{CD}
\widetilde{\mathit{PSL}}(2,\R)  @>{\subset}>> \widetilde{\mathrm{Diff}}_+^\omega S^1 @>{\subset}>>\Gamma_H \\
 @V{\varphi_n}VV   @V{\varphi_n}V{}V  @V{\varphi_n}V{}V\\ 
\widetilde{\mathit{PSL}}(2,\R)^{(n)} @>{\subset}>> \widetilde{\mathrm{Diff}}_+^\omega S^1 @>{\subset}>> \Gamma_H.
\end{CD}
$$
Namely the universal cover $\widetilde{\mathit{PSL}}(2,\R)^{(n)}$ of $\mathit{PSL}(2,\R)^{(n)}$, considered as a subgroup of
$\widetilde{\mathrm{Diff}}_+^\omega S^1$, is equal to the image under $\varphi_n$ of $\widetilde{\mathit{PSL}}(2,\R)$
(cf. \cite{morita}). Keeping this in mind, we define
$$
\widetilde{\mathit{PSL}}(2,\R)_{(2,3,n)}=\widetilde{\mathit{PSL}}(2,\R)*_{\R} (\R\rtimes Z_{(2,3,n)})
$$
where $Z_{(2,3,n)}$ be the subgroup of $\R^+$ generated by three elements $S_2,S_3,S_n$
(if $n$ does not have prime factor other than $2,3$, then $Z_{(2,3,n)}=Z_{(2,3)}$. But this does
not affect the following argument).
Since the action of $\varphi_n$ is equal to the conjugation by the element $S_n^{-1}$,
we can conclude that $\widetilde{\mathit{PSL}}(2,\R)^{(n)}$ is contained in the group $\widetilde{\mathit{PSL}}(2,\R)_{(2,3,n)}$.

\begin{prop}
If $n$ has a prime factor other than $2,3$, then we have 
$$
H_3(\R\rtimes Z_{(2,3,n)};\Z)\cong  \Z\oplus H_2(\R;\Z)
\cong \Z\oplus H_3(\R\rtimes Z_{(2,3)};\Z).
$$
\label{prop:z}
\end{prop}
\begin{proof}
The assumption implies that $Z_{(2,3,n)}\cong\Z^3$.
Then the argument in the proof of Proposition \ref{prop:b} is valid except for the term
$E^2_{3,0}=H_3(Z_{(2,3,n)};H_0(\R;\Z))\cong H_3(\Z^3;\Z)\cong\Z$ so that
$H_3(\R\rtimes Z_{(2,3,n)};\Z)\cong\Z$. Thus we have the following short
exact sequence (compare Equation \eqref{eq:ses})
$$
0\rightarrow H_3(\R\rtimes Z_{(2,3,n)};\Z)\cong\Z\rightarrow
H_3(\R\rtimes Z_{(2,3,n)},\R;\Z)\overset{\partial}{\rightarrow} H_2(\R;\Z)\rightarrow 0.
$$
There is a right inverse to the boundary homomorphism $\partial$ which is 
given by 
$$
H_2(\R;\Z)\overset{\Psi}{\rightarrow} H_3(\R\rtimes Z_{(2,3)},\R;\Z)
\overset{\subset}{\rightarrow} H_3(\R\rtimes Z_{(2,3,n)},\R;\Z)
$$
which we denote simply by $\Psi$.
The result follows.
\end{proof}

\begin{cor}
The right inverse to the boundary homomorphism 
$$
\partial: H_3(\R\rtimes Z_{(2,3,n)},\R;\Z)\rightarrow H_2(\R;\Z)
$$
is uniquely defined which is $\Psi$ in the above proposition.
\label{cor:unique}
\end{cor}
\begin{proof}
Let $\Psi'$ be any right inverse to $\partial$. Then $\Psi'-\Psi$ is a homomorphism
$H_2(\R;\Z)\rightarrow \mathrm{Ker}\,\partial\cong\Z$. Since $H_2(\R;\Z)$ is a vector space
over $\Q$ (or uniquely divisible group), there is no non-trivial homomorphism to $\Z$.
Hence $\Psi'=\Psi$.
\end{proof}

Based on these facts, we make the following definition of
our real analytic lift of Tsuboi's homomorphism $t_n$.

\begin{definition}
We define a homomorphism
$$
\Phi_n: H_2(\mathit{PSL}(2,\R)^{(n)},\mathit{SO}(2);\Z)\cong \R \rightarrow H_3(\widetilde{\mathit{PSL}}(2,\R)_{(2,3,n)};\Z)\rightarrow H_3(\Gamma_H;\Z)
$$
as follows. For any relative homology class $\sigma\in H_2(\mathit{PSL}(2,\R)^{(n)},\mathit{SO}(2);\Z)$
we set
$$
\Phi_n(\sigma)=\tilde{\sigma}-\Psi(\partial\tilde{\sigma})\in 
H_3(\widetilde{\mathit{PSL}}(2,\R)_{(2,3,n)};\Z)\ \text{and then in}\  
H_3(\Gamma_H;\Z).
$$
\label{def:Phin}
\end{definition}


\begin{prop}
The above definition is well defined.
\label{prop:Phin}
\end{prop}

\begin{proof}
As in the cases of Proposition \ref{prop:Phi} and Proposition \ref{prop:Phi2},
we have to check that the homology class of the cycle $\tilde{\sigma}-\Psi(\partial\tilde{\sigma})$
is uniquely determined independent of various choices made.
The long exact sequence of the homology of the pair
$(\widetilde{\mathit{PSL}}(2,\R)_{(2,3,n)},\R)$
$=(\widetilde{\mathit{PSL}}(2,\R)*_\R \R\rtimes Z_{(2,3,n)},\R)$ is given by
$$
\cdots\rightarrow H_3(\R;\Z)\overset{\text{$0$-map}}{\rightarrow}
H_3(\widetilde{\mathit{PSL}}(2,\R)_{(2,3,n)};\Z)
\rightarrow H_3(\widetilde{\mathit{PSL}}(2,\R)_{(2,3,n)},\R;\Z)\overset{\partial}{\rightarrow} H_2(\R;\Z) \rightarrow \cdots.
$$
By the excision, we have
$$
H_3(\widetilde{\mathit{PSL}}(2,\R)_{(2,3,n)},\R;\Z)\cong H_3(\widetilde{\mathit{PSL}}(2,\R),\R;\Z)\oplus
H_3(\R\rtimes Z_{(2,3,n)},\R;\Z).
$$
Therefore, we can conclude that
\begin{equation}
\begin{split}
H_3(\widetilde{\mathit{PSL}}&(2,\R)_{(2,3,n)};\Z)\cong \\
\mathrm{Ker}&\left(\partial: H_3(\widetilde{\mathit{PSL}}(2,\R),\R;\Z)\oplus
H_3(\R\rtimes Z_{(2,3,n)},\R;\Z)
\rightarrow H_2(\R;\Z)\right).
\end{split}
\label{eq:kern}
\end{equation}
The extra $\Z$-summand in Proposition \ref{prop:z} vanishes under $\partial$.
Hence the right hand side of \eqref{eq:kern} is isomorphic to
$$
\Z\oplus \mathrm{Ker}\left(\partial: H_3(\widetilde{\mathit{PSL}}(2,\R),\R;\Z)\oplus
H_3(\R\rtimes Z_{(2,3)},\R;\Z)
\rightarrow H_2(\R;\Z)\right)
$$
which in turn is isomorphic to 
$$
\Z\oplus H_3(\widetilde{\mathit{PSL}}(2,\R)_{(2,3)};\Z)
$$
by Proposition \ref{prop:b}.
Hence the difference
$\tilde{\sigma}-\Psi(\partial\tilde{\sigma})$ defines a unique element
in \\$H_3(\widetilde{\mathit{PSL}}(2,\R)_{(2,3,n)};\Z)$ by Equality \eqref{eq:kern}.
This completes the proof.

\end{proof}

\begin{remark}
The above proposition shows that the extra $\Z$-summand in Proposition \ref{prop:z}
does not affect our conclusion. However we will see in the next section $\S 8$ that
this $\Z$-summand does vanish in $H_3(\Gamma_H;\Z)$.
\end{remark}

\vspace{2mm}
\noindent
{\it Proof of Theorem \ref{th:mainb}}\\


Since we have 
$$
H_2(G_T,\mathit{SO}(2);\Z)\cong \bigoplus_{n=1}^\infty H_2(\mathit{PSL}(2,\R)^{(n)},\mathit{SO}(2);\Z),
$$
in order to prove Theorem \ref{th:mainb}, it is now enough to prove the following statement which is a generalization of
Proposition \ref{prop:comp} to the case of $n$-fold covering.
\begin{prop}
For any parameter $s\in\R$, we have the following identity
$$
\Phi_n(s)=n^2 \Phi(s)
$$
where $\Phi(s)$ denotes $\Phi(\sigma)$ such that the parameter of $\sigma$ is equal to $s$ and similarly
for $\Phi_n$.
\label{prop:compn}
\end{prop}
\begin{proof}
Since
$$
\varphi_n(1\wedge s)=\frac{1}{n}\wedge \frac{s}{n}=\frac{1}{n^2} 1\wedge s,
$$
we have
$$
n^2\varphi_n(\Phi(s))=\Phi_n(s).
$$
Here we need Corollary \ref{cor:unique} to show this. Uniqueness guarantees the fact that the
action of $\varphi_n$ on the right is trivial. 
Formerly in the case of $n=2$ (Proposition \ref{prop:comp}),
there is no need of this because in that case the uniqueness of the right inverse was obvious since $\Psi$ was an
isomorphism.
On the other hand, $\varphi_n$ induces an inner automorphism of the target group $\widetilde{\mathit{PSL}}(2,\R)_{(2,3,n)}$.
Therefore $\varphi_n(\Phi(s))=\Phi(s)$. The required result follows from this.
\end{proof}
This completes the proof of Theorem \ref{th:mainb}.
\qed

\section{Uniqueness of our construction within the affine group}

In section $\S 5$, we construct what we call real analytic caps
(see Definition \ref{def:ac}) which are $3$-chains bounding $2$-cycles representing linear foliations
on the torus whose monodromy groups are contained in the group
$\R$ acting on itself by translations.
They play the role of Reeb components
in the real analytic context.
In this section, we show that our construction of real analytic caps
are uniquely defined within the affine group $\mathrm{Aff}^+ (1)=\R\rtimes \R^+$
(see Proposition \ref{prop:unique} below).

We mentioned in Remark \ref{re:af} that K. Rozhe (C. Roger)  \cite{rozhe}
proved that the natural projection $\mathrm{Aff}^+ (n)\rightarrow \mathrm{GL}^+ (n,\R)$
induces an isomorphism
$$
H_*(\mathrm{Aff}^+ (n);\Z)\cong H_*(\mathrm{GL}^+ (n,\R);\Z)
$$
for any $n$.
In particular, the homomorphism
$
H_*(\R^n ;\Z)\rightarrow H_*(\mathrm{Aff}^+ (n);\Z)
$
is trivial except for the degree $0$ part. 
In the case of $n=1$,
the affine group $\mathrm{Aff}^+ (1)=\R\rtimes \R^+$ is a subgroup of
$\mathrm{Diff}_+^\omega \R$ so that it can also be considered as a
subgroup of the Haefliger group $\Gamma_H$.
More precisely, 
$
\R=\{T_s; s\in \R\},\R^+=\{S_r; r\in \R^+\}\subset \mathrm{Diff}_+^\omega \R
$
where $T_s(x)=x+s, S_r(x)=rx \ (x\in\R)$
(see $\S 3$).
The following Proposition exhibits a certain ``flexibility" of the Haefliger group.

\begin{prop}
If we consider the two groups
$
\R=\{T_s; s\in \R\},\R^+=\{S_r; r\in \R^+\}\subset \mathrm{Diff}_+^\omega \R
$
as subgroups of the Haefliger group $\Gamma_H$, then they are
conjugate to each other.
\label{prop:th}
\end{prop}

\begin{proof}
Consider the exponential mapping
$$
\exp: \R\rightarrow (0,\infty)
$$
which is a real analytic isomorphism. Under this isomorphism, the transformations of
$T_s (\text{parallel transformation}), S_r (\text{homotheties}) \in \R\rtimes\R^+$ 
acting on the real line $\R$ on the left hand side are transferred to actions
on the righthand side $(0,\infty)$ as follows.
\begin{align*}
&(0,\infty)\ni x \overset{\log}{\rightarrow}\log x \overset{T_s}{\rightarrow} \log x+s
\overset{\exp}{\rightarrow} x\cdot e^s=R_{e^s}(x)\\
&(0,\infty)\ni x \overset{\log}{\rightarrow}\log x \overset{S_r}{\rightarrow} r\log x=x^r
\overset{\exp}{\rightarrow} x^r
\end{align*}
Thus $T_s$ is transferred to the homothety $R_{e^s}$. Here the action of $R_{e^s}$ is 
defined only on $(0,\infty)$. However, it is analytically continued to the whole real line 
as a global homothety there. In view of the descriptions of elements of $\Gamma_H$
in \cite{haefliger}\cite{jekel}, we see that $T_s$ is conjugate to $R_{e^s}$ by the 
mapping $\exp$ considered as an element of $\Gamma_H$. Since this holds for any $s$,
this finishes the proof.
\end{proof}

\begin{prop}
Let
$$
i: \R\rtimes\R^+\subset \mathrm{Diff}_+^\omega\R\subset \pi_1(B\overline{\Gamma}^\omega_1)=\Gamma_H
$$
be the natural inclusion. Then the induced homomorphism
$$
i_*: H_*(\R\rtimes\R^+;\Z)\rightarrow H_*(B\overline{\Gamma}^\omega_1;\Z)
$$
is trivial in positive degrees. 
\label{prop:ac}
\end{prop}

\begin{proof}
According to a result of Rozhe cited above, the natural projection
$\R\rtimes\R^+\rightarrow \R^+$ induces
an isomorphism $H_*(\R\rtimes\R^+;\Z)\cong H_*(\R^+;\Z)$.
Therefore it is enough to prove that the homomorphism
\begin{equation}
H_*(\R^+;\Z)\rightarrow H_*(\Gamma_H;\Z)
\label{eq:r+}
\end{equation}
induced by the inclusion $\R^+\subset \Gamma_H$ is trivial in positive degrees.
By Proposition \ref{prop:th}, the subgroup $\R^+\subset \Gamma_H$ is conjugate to the
subgroup $\R\subset \Gamma_H$. Hence the homomorphism \eqref{eq:r+}
is trivial in positive degrees 
if and only if the same is true for the subgroup $\R\subset \Gamma_H$.
Rozhe's result above implies that $H_*(\R;\Z)$ vanishes, in positive degrees,
already on $H_*(\R\rtimes\R^+;\Z)$ so that on $H_*(\Gamma_H;\Z)$ as well. This completes the proof.
\end{proof}

\begin{question}
Obviously the above Proposition \ref{prop:th} does not hold if we replace $\Gamma_H$ with
$\mathrm{Diff}_+^{\omega} \R$. However, it seems reasonable to conjecture that
the natural homomorphism
$$
H_*(\R\rtimes\R^+;\Z)\cong H_*(\R^+;\Z)\rightarrow H_*(\mathrm{Diff}_+^{\omega} \R;\Z)
$$
is trivial in positive degrees. Is this correct? We mention here that if we replace 
$H_*(\mathrm{Diff}_+^{\omega} \R;\Z)$ with $H_*(\mathrm{Diff}_+^{\infty} \R;\Z)$
in the above homomorphism,
then the above statement holds. This follows by combining Proposition \ref{prop:ac}
with a theorem of Segal which says
$H_*(\mathrm{Diff}_+^{\infty} \R;\Z)\cong H_*(B\overline{\Gamma}_1^\infty;\Z)$
(\cite{segal}).
\end{question}

\begin{question}
It seems reasonable to conjecture 
that the natural mapping
$$
B(\R\rtimes\R^+)\rightarrow B\overline{\Gamma}_1^\infty
$$
is homotopic to a constant map. Is this correct? Here recall that Tsuboi \cite{tsuboi84b} proved that
the natural map $B\R^+\rightarrow B\overline{\Gamma}_1^\infty$ is homotopic to a constant map
(this is a special case of his more general result).
\end{question}

With above preparations in mind, we proceed to state the main result of this section.
Our real analytic caps was a chain-level realization of the homomorphism
$$
\Psi:H_2(\R;\Z)\overset{\cong}{\rightarrow} H_3(\R\rtimes Z_{(2,3)},\R;\Z)
$$
which is the inverse of the boundary homomorphism 
$\partial:H_3(\R\rtimes Z_{(2,3)},\R;\Z)\overset{\cong}{\rightarrow} H_2(\R;\Z)$.
We consider what will happen if we replace $Z_{(2,3)}$ with the group $\R^+$
of all the homotheties. 

\begin{prop}
The long exact sequence of the pair $(\R\rtimes \R^+,\R)$ gives rise to
the following short exact sequence
$$
0 \rightarrow H_3(\R\rtimes \R^+;\Z)\cong H_3(\R^+;\Z)\rightarrow
H_3(\R\rtimes \R^+,\R;\Z)
\overset{\partial}{\rightarrow} H_2(\R;\Z)\rightarrow 0.
$$
\label{prop:afd}
\end{prop}

\begin{proof}
This follows from the result of Rozhe (Roger) cited above.
The fact that the homomorphism
$H_k(\R;\Z)\rightarrow H_k(\R\rtimes \R^+;\Z)$ is trivial for $k=2,3$ was
proved in $\S 5$ explicitly.
\end{proof}

Since both of the groups $H_3(\R^+;\Z)$ and $H_2(\R;\Z)$ are $\Q$ vector spaces,
$H_3(\R\rtimes \R^+,\R;\Z)$ is also a $\Q$ vector space.
Therefore the boundary operator $\partial:H_3(\R\rtimes \R^+,\R;\Z)
\rightarrow H_2(\R;\Z)$ has a huge amount of right inverses among which our homomorphism
$\Psi:H_2(\R;\Z)\rightarrow H_3(\R\rtimes Z_{(2,3)},\R;\Z)\subset H_3(\R\rtimes \R^+,\R;\Z)$
is the simplest one. Associated to any choice $\Psi':H_2(\R;\Z)\rightarrow H_3(\R\rtimes \R^+,\R;\Z)
$ of right inverse to this boundary operator $\partial$,
we can define an analytic cap and then the homomorphisms
$$
\Phi'_n: H_2(\mathit{PSL}(2,\R)^{(n)},\mathit{SO}(2);\Z)\cong \R \rightarrow H_3(\widetilde{\mathit{PSL}}(2,\R)_{(2,3,n)};\Z)\rightarrow H_3(\Gamma_H;\Z)\quad
$$
for any $n=1,2,\ldots$.
However, we have the following uniqueness result.

\begin{prop}
If we consider the targets of our homomorphisms 
$$
\Phi_n: H_2(\mathit{PSL}(2,\R)^{(n)},\mathit{SO}(2);\Z)\cong \R 
\rightarrow H_3(\Gamma_H;\Z)\quad
(n=1,2,\ldots)
$$
to be $H_3(\Gamma_H;\Z)$ rather than $H_3(\widetilde{\mathit{PSL}}(2,\R)_{(2,3,n)};\Z)$, then they
are independent of choices of the right inverse to the boundary operator so that it is uniquely defined.
\label{prop:unique}
\end{prop}

\begin{proof}
Let $\Psi,\Psi':H_2(\R;\Z)\rightarrow H_3(\R\rtimes \R^+,\R;\Z)$
be two choices of right inverses of the boundary operator $\partial:H_3(\R\rtimes \R^+,\R;\Z)
\rightarrow H_2(\R;\Z)$. Then their difference defines a homomorphism
$$
\Psi'-\Psi: H_2(\R;\Z)\rightarrow H_3(\R\rtimes \R^+;\Z).
$$
Hence we see that the difference of two choices of associated real analytic caps is measured by the group
$H_3(\R\rtimes \R^+;\Z)$. However by Proposition \ref{prop:ac},
we know that $H_3(\R\rtimes \R^+;\Z)$ vanishes in $H_3(\Gamma_H;\Z)$.
The required result follows from this, completing the proof.
\end{proof}

\section{A new kind of characteristic class of foliations}

In this section, we prove Theorem \ref{th:curious}.

We begin by recalling an old theorem of Wiegold in \cite{wiegold} which was published in the late $1960$'s.

\begin{thm}[Wiegold \cite{wiegold}]
There exists an isomorphism
$$
\mathrm{Ext}(\Q,\Z)\cong \R
$$
as an additive group.
\label{th:wiegold}
\end{thm}

To be a bit more precise, let us consider
the central extension
$$
0\rightarrow \Z\rightarrow\Q\rightarrow \Q/\Z\rightarrow 1
$$
which yields the Gysin exact sequence
$$
H^1(\Q;\Z)=0\rightarrow H^0(\Q/\Z;\Z)\rightarrow H^2(\Q/\Z;\Z)\rightarrow H^2(\Q;\Z)\rightarrow H^1(\Q/\Z;\Z)=0.
$$
By the universal coefficient theorem, we see that
$$
H^2(\Q/\Z;\Z)\cong \mathrm{Ext}(\Q/\Z,\Z),\quad H^2(\Q;\Z)\cong \mathrm{Ext}(\Q,\Z).
$$
Wiegold showed in the above cited paper that
$$
\mathrm{Ext}(\Q/\Z,\Z)\cong\hat{\Z},\quad \mathrm{Ext}(\Q,\Z)\cong \R
$$
which gives rise to a short exact sequence
$$
0\rightarrow \Z\rightarrow\hat{\Z}\rightarrow \R\rightarrow 0.
$$
We would like to quote the final sentence of the above cited paper:
``I would not have imagined that killing a cyclic subgroup could
have so profound an effect".

\begin{prop}
For any n, we have an isomorphism
$$
H^{n+1}(K(\R,n);\Z)\cong \mathrm{Ext}(\R,\Z)\cong \prod_{\lambda\in\mathcal{H}} \R_\lambda
$$
where the meaning of the symbol $\prod_{\lambda\in\mathcal{H}} \R_\lambda$ is described in
Theorem \ref{th:curious}.
\label{prop:kr}
\end{prop}
\begin{proof}
By the universal coefficient theorem for cohomology, we have a short exact sequence
$$
0\rightarrow \mathrm{Ext}(H_n(K(\R,n),\Z),\Z)\rightarrow H^{n+1}(K(\R,n);\Z)\rightarrow
\mathrm{Hom}(H_{n+1}(K(\R,n),\Z);\Z)\rightarrow 0.
$$
Since $H_{n+1}(K(\R,n);\Z)=0$ for any $n>1$ and $H_{2}(K(\R,1);\Z)$ is a $\Q$ vector space,
we have $\mathrm{Hom}(H_{n+1}(K(\R,n),\Z),\Z)=0$ for any $n$. Therefore we can conclude that
$$
H^{n+1}(K(\R,n);\Z)\cong \mathrm{Ext}(H_n(K(\R,n),\Z),\Z) \cong \mathrm{Ext}(\R,\Z).
$$
Now we can write
$$
\R\cong \sum_{\lambda\in\mathcal{H}} \Q_\lambda
$$
as a $\Q$ vector space. Hence
\begin{align*}
\mathrm{Ext}(\R;\Z)&\cong \mathrm{Ext}\left(\sum_{\lambda\in\mathcal{H}} \Q_\lambda,\Z\right)\\
&\cong \prod_{\lambda\in\mathcal{H}}  \mathrm{Ext}(\Q_\lambda,\Z)\quad (\text{by the property of $\mathrm{Ext}$})\\
&\cong \prod_{\lambda\in\mathcal{H}} \R_\lambda\quad (\text{by Theorem \ref{th:wiegold}}).
\end{align*}
This completes the proof.
\end{proof}

\noindent
{\it Proof of Theorem \ref{th:curious}}\\
Let $gv: B\overline{\Gamma}_1^\omega\rightarrow K(\R,3)$ be the classifying map of the
Godbillon-Vey class. Our task is to show that the induced homomorphism
$$
gv^*: H^4(K(\R,3),\Z)\cong \prod_{\lambda\in\mathcal{H}} \R_\lambda\rightarrow 
H^4(B\overline{\Gamma}_1^\omega,\Z)
$$
is injective.

By the universal coefficient theorem for cohomology, we have the following split short exact sequence
$$
0\rightarrow \mathrm{Ext}(H_3(B\overline{\Gamma}_1^\omega,\Z),\Z)\rightarrow H^{4}(B\overline{\Gamma}_1^\omega;\Z)\rightarrow
\mathrm{Hom}(H_{4}(B\overline{\Gamma}_1^\omega,\Z);\Z)\rightarrow 0.
$$

By comparing this with the corresponding short exact sequence for the space $K(\R,3)$, keeping
the above Proposition \ref{prop:kr} in mind, we obtain
the following commutative diagram
\begin{equation}
\begin{CD}
 \mathrm{Ext}(H_3(K(\R,3),\Z),\Z) @>{\cong}>> H^{4}(K(\R,3);\Z) @>>> \mathrm{Hom}(H_{4}(K(\R,3),\Z);\Z)=0\\
 @V{gv^*}VV   @V{gv^*}VV  @VVV\\ 
  \mathrm{Ext}(H_3(B\overline{\Gamma}_1^\omega,\Z),\Z)
   @>{\subset}>> H^{4}(B\overline{\Gamma}_1^\omega;\Z) @>>> \mathrm{Hom}(H_{4}(B\overline{\Gamma}_1^\omega,\Z);\Z).
\end{CD}
\label{eq:cd3}
\end{equation}

Now as was already mentioned in the Introduction, Thurston proved that the homomorphism
$$
gv: H_3(B\overline{\Gamma}_1^\omega;\Z)\rightarrow \R
$$
is surjective. On the other hand, by our main Theorem \ref{th:main}, this homomorphism splits
so that there exists a direct summand $\R\subset H_3(B\overline{\Gamma}_1^\omega;\Z)$ such that
$$
H_3(B\overline{\Gamma}_1^\omega;\Z)= \R\oplus \mathrm{Ker}\, gv=H_3(K(\R,3);\Z)\oplus \mathrm{Ker}\, gv.
$$
By the property of the functor $\mathrm{Ext}$, we have
$$
\mathrm{Ext}(H_3(B\overline{\Gamma}_1^\omega;\Z),\Z)\cong \mathrm{Ext}(H_3(K(\R,3);\Z),\Z)\oplus \mathrm{Ext}(\mathrm{Ker}\, gv,\Z).
$$
If we put this into the above commutative diagram \eqref{eq:cd3}, we obtain the required result.
This completes the proof of the former part of our theorem.

Next we show the following facts to prove the latter claims of our Theorem. 
First, let $X$ be any CW complex such that the homology group $H_3(X;\Z)$ is 
finitely generated. Then for any
continuous map $f: X\rightarrow B\overline{\Gamma}_1^\omega$, the induced homomorphism
$$
f^*: H^4(B\overline{\Gamma}_1^\omega;\Z)\rightarrow H^4(X,\Z)
$$
is trivial on $\mathrm{Im}\, (gv^*: H^4(K(\R,3);\Z)\rightarrow H^4(B\overline{\Gamma}_1^\omega;\Z))$.
To prove this, it is enough to show that for any continuous map $h: X\rightarrow K(\R,3)$, the induced homomorphism
$$
h^*:H^4(K(\R,3);\Z)\rightarrow H^4(X;\Z)
$$
is trivial. By the proof of Proposition \ref{prop:kr}, this is equivalent to show that the homomorphism
$$
h^*:\mathrm{Ext(H_3(K(\R,3);\Z),\Z)}\rightarrow \mathrm{Ext(H_3(X;\Z),\Z)}
$$
is trivial. By the assumption, $H_3(X;\Z)$ is a finitely generated abelian group.
As is well-known, $\mathrm{Ext}(\Z,\Z)=0$ and  $\mathrm{Ext}(\Z/n,\Z)\cong \Z/n$ so that $\mathrm{Ext(H_3(X;\Z),\Z)}$
is a finite abelian group. On the other hand, 
the source group of $h^*$ is $\mathrm{Ext(H_3(K(\R,3);\Z),\Z)}$ which is isomorphic
to $\prod_{\lambda\in\mathcal{H}} \R_\lambda$ essentially by the Theorem of Wiegold \cite{wiegold}.
We can now conclude that the homomorphism $h^*$ is trivial. This completes the proof.

Next assume that there is given a $\overline{\Gamma}_1^\omega$-structure on $X$ and there exists a 
uniquely divisible subgroup $V\subset H_3(X;\Z)$ such that the homomorphism $gv:H_3(X;\Z)\rightarrow \R$
induced by the Godbillon-Vey class is non-trivial on $V$. By the assumption, $V$ is uniquely divisible
so that it is a $\Q$ vector space. Then, since the target of $gv$ is also a $\Q$ vector space, it is easy to see
that $gv|_{V}$ is a $\Q$ linear homomorphism. It follows that $V\cong \left(\mathrm{Ker}\, gv\cap V\right )\oplus W$
for some subgroup $W\subset H_3(X;\Z)$ such that $gv: W\cong \mathrm{Im}\, gv\subset \R$
(here we use the Axiom of Choice if $\dim_\Q V=\infty$).
This then implies that
$$
H_3(X;\Z)= \mathrm{Ker}\, gv\oplus W.
$$
If we use a similar argument as in the former part of this proof by replacing $\R$
with its subspace $\mathrm{Im}\, gv(\cong W) \subset \R$, then we can deduce that
the homomorphism
$$
gv^*: H^4(\mathrm{K(Im}\, gv,3);\Z) \cong \prod_{\lambda\in\mathcal{H}'} \R_\lambda\rightarrow 
H^4(X,\Z)
$$
is injective. Here $\mathcal{H}'$ denotes a Hamel basis for the $\Q$ subspace $\mathrm{Im}\, gv\, (\cong W) \subset \R$.
Since the homomorphism $H^4(K(\R,3);\Z)\rightarrow H^4(\mathrm{K(Im}\, gv,3);\Z)$
induced by the inclusion $\mathrm{Im}\, gv\subset \R$ is surjective,
we can conclude that our characteristic classes detect the subgroup $W\subset H_3(X;\Z)$ on which the Godbillon-Vey class
gives an injection into its target $\R$ as required.

This completes the proof.
\qed

\bibliographystyle{amsplain}

\end{document}